\documentclass[11pt,a4paper,leqno]{article}

\usepackage{latexsym,amsfonts,amssymb,euscript,amsthm}
\usepackage{amsmath}
\usepackage{paralist}
\usepackage{graphicx}
\usepackage{subcaption}
\usepackage{graphics} 
\usepackage{epsfig} 
\usepackage{graphicx}
\usepackage{epstopdf} 
\usepackage{float}
\usepackage[colorlinks=true]{hyperref}
\usepackage{imakeidx}
\usepackage{mathtools}
\usepackage{authblk}
\usepackage[backend=biber, style=numeric]{biblatex}
\mathtoolsset{showonlyrefs}
\usepackage{color}
\hypersetup{urlcolor=blue, citecolor=red}

\newtheorem{theorem}{Theorem}[section]
\newtheorem{corollary}{Corollary}

\newtheorem{lemma}[theorem]{Lemma}
\newtheorem{proposition}{Proposition}

\theoremstyle{definition}

\newtheorem{remark}{Remark}

\newcommand{\R} {\mathbb{R}}

\newcommand{\eps}{\epsilon}
\newcommand{\eto}{\stackrel{\eps\to 0}{\longrightarrow}}
\newcommand{\weto}{\stackrel{\eps\to 0}{\rightharpoonup}}
\def\eto{\buildrel \epsilon\to 0\over\longrightarrow }

\makeindex[intoc]
\title{Elliptic Equations in Weak Oscillatory Thin Domains: Beyond Periodicity with Boundary-Concentrated Reaction Terms}

\date{}

\author[1]{Pricila S. Barbosa\thanks{e-mail: pricilabarbosa@utfpr.edu.br}}

\author[2]{Manuel Villanueva-Pesqueira\thanks{e-mail: mvillanueva@comillas.edu. Partially supported by PID2022-137074NB-I00 funded by MCIN}}

\affil[1]{ Departamento de Matem\'atica, Universidade Tecnológica Federal do Paraná, Londrina, Brazil}

\affil[2]{Grupo de Dinámica No Lineal, Universidad Pontificia Comillas, ICAI, Alberto Aguilera 25
	28015 Madrid, Spain}

\begin{document}
\maketitle
\begin{abstract}
In this paper we analyze the limit behavior of a family of solutions of the Laplace operator with homogeneous Neumann boundary conditions, set in a two-dimensional thin domain which presents weak oscillations on both  boundaries and  with terms concentrated in a narrow oscillating neighborhood of the top boundary. 
The aim of this problem is to study the behavior of the solutions as the thin domain presents oscillatory behaviors beyond the classical periodic assumptions,including scenarios like quasi-periodic or almost-periodic oscillations. We then prove that the family of solutions converges to the solution of a $1-$dimensional limit equation capturing the geometry and oscillatory behavior of  boundary of the domain and the narrow strip where the concentration terms take place. In addition, we include a series of numerical experiments illustrating the theoretical results obtained in the quasi-periodic context.

\end{abstract}
\section{Introduction}

We are keen on examining the behavior of the solutions to certain Elliptic Partial Differential Equations which are posed in a oscillating thin domain $R^\epsilon$. This domain is a slender region in $\R^{2}$ displaying oscillations along its boundary. It is described as the region between two oscillatory functions, that is, 
\begin{equation}\label{thin-appen}
R^\epsilon = \Big\{ (x, y) \in \R^{2} \; | \;  x \in I \subset \R,  \; -\eps k^1(x,\eps)< y < \epsilon k^2(x,\eps) \Big\},
 \end{equation}
 where $I=[a,b]$ is any interval in $\mathbb{R}$ and the functions $k^1,k^2$ satisfy certain hypotheses, see {\bf (H)}  in the next section. These two functions may present periodic oscillation, see Figure \ref{thin1 appen} for a example, but also more intricate situations beyond the classical periodic setting, see Figure \ref{quasithin} for an example with quasi-periodic oscillating boundaries. 

 \begin{figure}[H]
  \centering
  \includegraphics [width=8cm, height=4cm]{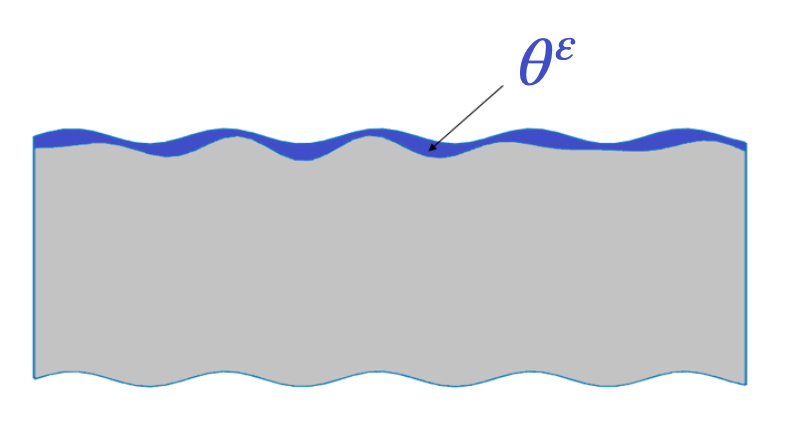}
   \caption{Thin domain $R^\eps$ with doubly periodic weak oscillatory boundary.}
   \label{thin1 appen}
\end{figure}

\begin{figure}[H]
  \centering
  \includegraphics [width=8cm, height=3.5cm]{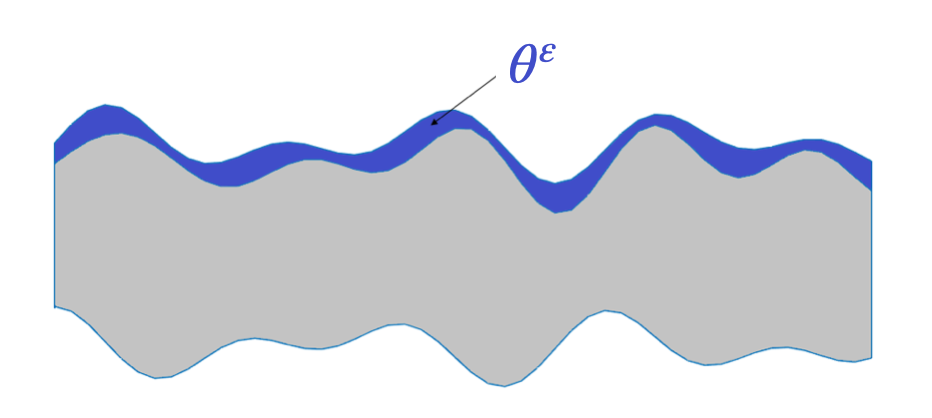}
   \caption{Thin domain $R^\eps$ with doubly quasi-periodic oscillatory boundary.}
   \label{quasithin}
\end{figure}

We are interested in analyzing the following linear elliptic equation with reaction terms concentrated in a very narrow oscillating neighborhood of the upper oscillating boundary
\begin{equation} \label{1OPI0}
\left\{
\begin{gathered}
- \Delta w^\epsilon + w^\epsilon = \frac{1}{\epsilon^\gamma} \chi_{\theta^\epsilon} f^\epsilon
\quad \textrm{ in } R^\epsilon, \\
\frac{\partial w^\epsilon}{\partial \nu^\epsilon} = 0
\quad \textrm{ on } \partial R^\epsilon,
\end{gathered}
\right. 
\end{equation}
where $f^\epsilon \in L^2(R^\epsilon)$, 
and $\nu^\epsilon$ is the unit outward normal to $\partial R^\epsilon$. Let be $\chi_{\theta^\epsilon}$ the characteristic function of the narrow strip defined by 
\begin{equation}\label{theta}
   \theta^\epsilon = \Big\{ (x, y) \in \R^{2} \; | \;  x \in I \subset \R,  \; \epsilon[ k^2(x,\eps)  - \epsilon^\gamma H_\eps(x)] < y < \epsilon k^2(x,\eps)  \Big\}, 
\end{equation}
 where $\gamma > 0$ and $H_\eps: I \subset \R \to \R$ is a $C^1$ quasi-periodic and bounded function.

It is important to highlight that we will obtain the limiting problem explicitly, assuming very general conditions on oscillatory functions. Afterward, we will study this limiting problem in interesting specific cases, such as quasi-periodic and even almost periodic functions. The authors firmly believe that analyzing problems in homogenization beyond the classical periodic setting is compelling from both an applied and a purely mathematical standpoint. In the real world, many materials and structures defy perfect periodicity and many real world problems involve multiple scales of heterogeneity or complexity.

 Homogenization, a powerful mathematical tool in the realm of materials science and engineering, has traditionally been confined to the study of materials with perfect periodic structures, see \cite{All,BenLioPa,CioDamGri,CioPau}. However, a growing body of research suggests that venturing beyond this classical periodic setting offers a rich landscape of possibilities, both in terms of practical applications and as a playground for mathematical exploration. Therefore, over the years, different homogenization techniques have been adapted and created to study problems beyond the periodic setting. In this context, the almost periodic framework offers an appropriate answer to certain limitations presented by periodic models.
To illustrate how this topic is at the forefront of research, we can cite several articles, knowing that there are many more, arranged from the most classic to the most recent, see \cite{ArmGlo, Bra, Koz,OleiZi,Jin}.

In the literature, there exists a multitude of studies dedicated to exploring the impact that oscillatory boundaries and thickness of the domain have on the solution behaviors of partial differential equations posed in thin domains. For further insight into this subject, one may refer to the following selected references and the studies cited within them \cite{,ArrVi2017,ArrCarPerSil,,ArrPer2013, HR,R}. However, the volume of research that avoids the consideration of periodic boundaries is significantly less. In this regard,
it is worth mentioning some articles have dealt with bounded domains with locally periodic boundaries, see \cite{ArrPer2011,ArrVil2016}. 
Equations set in thin domains with the upper and lower boundaries oscillating at different periods have also been studied. This means that the requirements for traditional homogenization are not met, as there is no cell that describes the domain. In fact, the behavior of the two boundaries is reflected in the limit. While \cite{ArrVil2014a} focuses on boundaries with strong oscillations, \cite{ArrVi2020} explores weak oscillations, yielding results characteristic of quasi-periodic functions

On the other hand, numerous studies have also focused on singular elliptic and parabolic problems, characterized by potential and reactive terms that are localized within a narrow region close of the boundary. First, equations in fixed and bounded domains were studied, see \cite{AAA,JiBer2009,JiBer2011}. 
Recently, problems combining both singular situations have been studied. That is, thin domains with oscillating boundaries with terms concentrated on the boundary, \cite{AAM,AAM1,Kape}. However, all of these fall within the assumptions of periodicity.

The paper is organized as follows: In Section 2, we establish the necessary hypotheses, define the notation, and present the central result of the article, which reveals the explicit limit problem under general hypotheses about the oscillating boundaries. In Section 3, adapting the technique introduced in \cite{ArrVi2020}, we will show that after appropiate change of variables the solutions of the problem (\ref{1OPI0}) can be aproximated by a one-dimensional problem with oscillating coefficients and with a force term reflecting the reaction terms on the thin neighborhood of the upper boundary. In Section 4, we proceed to analyze the limit of the one-dimensional equation derived in the previous section, thereby proving the main result of the paper. 
In Section 5, we derive the limit problem for specific cases of oscillating functions, including quasi-periodic and almost-periodic types. Finally, Section 6, we show some numerical evidences about the proved results.

\par\bigskip\bigskip 
 \section{Assumptions, notations and main result}

Let $k^i$ be a function, $i=1,2$,
\begin{equation*}
\begin{array}{rccl}
k^i: &I\times (0,1) &\longrightarrow &\R^+ \\
 &(x,\eps)&\longrightarrow &k^i(x,\eps)=k_\eps^i(x),
 \end{array}
 \end{equation*}
 such that
 \begin{itemize}
\item[\bf{(H.1)}] $k_\eps^i$ is a $C^1$ function and
\begin{equation}\label{main}
\eps\Big|\frac{d}{dx}k^i_\eps(x)\Big|\eto 0 \hbox{ uniformly in } I.
\end{equation}
\item[\bf{(H.2)}] There exist two positive constants independent of $\eps$ such that 
\begin{equation}\label{boundk}
0<C^i_1\leq k_\eps^i( \cdot) \leq C^i_2.
\end{equation}
\item[\bf{(H.3)}] There exists a function $K^i$ in $L^2(I)$ such that
$$k^i_\eps \weto K^i\quad \hbox{ w}-L^2(I).$$
\item[\bf{(H.4)}] There exists a function $P$ in $L^2(I)$ such that
$$\frac{1}{k^1_\eps+k^2_\eps} \weto P \quad \hbox{ w}-L^2(I).$$
\end{itemize}


 Indeed, the thickness of the domain has order $\eps$ and we say that the domain presents weak oscillations due to the convergence \eqref{main}.

 We will assume that the function which defines the narrow strip is a quasi-periodic smooth function. There are constants $H_0, H_1 \geq 0$ such that 
 \begin{equation}\label{est H}
 H_0 \leq H_\eps(x) \leq H_1 \,\,\mbox{for all}\,\, x \in I \,\,\mbox{and}\,\, \eps>0.
 \end{equation}
 
 First, we present the homogenized problem for the general case of functions satisfying the hypotheses {\bf (H)}. Subsequently, we direct our attention to several intriguing examples that provide insights into this matter. In particular, our overarching framework allows us consider quasi-periodic or almost periodic functions. See Figure \ref{quasithin} where  quasi-periodic functions are considered.  That is the case where  
 \begin{equation}\label{per}
 k^1_\eps(x)=h(x/\eps^\alpha), \quad k^2_\eps(x)=g(x/\eps^\beta),
 \end{equation}
with $0<\alpha, \beta<1$ and the functions $g,h \,: \R \to \R $ are $C^1$ quasi-periodic functions. 
Note that the extensively studied structure known as the periodic setting is included within this framework.

 The variational formulation of (\ref{1OPI0}) is the following:  find $w^\epsilon \in H^1(R^\epsilon)$ such that 
\begin{equation*} \label{VFP}
\int_{R^\epsilon} \Big\{ \displaystyle\frac{\partial w^\epsilon}{\partial x}\displaystyle\frac{\partial \varphi}{\partial x} + \displaystyle\frac{\partial w^\epsilon}{\partial y}\displaystyle\frac{\partial \varphi}{\partial y}
+ w^\epsilon \varphi \Big\} dx dy =  \frac{1}{\eps^\gamma}\int_{\theta^\epsilon} f^\epsilon \varphi dx dy, 
\, \forall \varphi \in H^1(R^\epsilon).
\end{equation*}

Observe that, for fixed $\eps>0$, the existence and uniqueness of solution to problem (\ref{1OPI0}) is guaranteed by Lax-Milgram Theorem. Then, we will analyze the behavior of the solutions as the parameter $\eps$ tends to zero. Particularly, by adapting the procedure demonstrated in \cite{ArrVi2020}, we initially implement a suitable change of variables. This enables us to substitute, in a certain sense, the original problem \eqref{1OPI0} posed in a 2-oscillating thin domain into a simpler problem with oscillating coefficients posed in an interval of $\mathbb{R}$. Notice that, this fact is in agreement with the intuitive idea that the family of solutions $w^\eps$ should converge to a function of just one variable as $\eps$ goes to zero since the domain shrinks in the vertical direction. Subsequently,  by using the previous results and adapting well-known techniques in homogenization we obtain explicitly the homogenized limit problem for the general case.

Due to the order of the height of the thin domains $R^\eps$ it makes sense to
 consider the following measure in thin domains
$$\rho_{\eps}(\mathcal{O}) = \frac{1}{\eps}|\mathcal{O}|, \; \forall\, \mathcal{O} \subset R^\eps.$$
 The rescaled Lebesgue measure $\rho_{\eps}$ allows
us to preserve the relative capacity of a measurable subset $\mathcal{O} \subset R^\eps$.
Moreover, using the previous measure we introduce the spaces $L^p( R^\eps, \rho_\eps)$ and  $W^{1,p}( R^\eps, \rho_\eps)$, for $1\leq p < \infty$
endowed with the norms obtained rescaling the usual norms by the factor $\frac{1}{\eps}$, that is, 
\begin{align*}
|||\varphi|||_{L^p(R^\eps)} &= \eps^{-1/p}||\varphi||_{L^p(R^\eps)},  \quad \forall \varphi \in L^p(R^\eps),\\
|||\varphi|||_{W^{1,p}( R^\eps)} &= \eps^{-1/p}||\varphi||_{W^{1,p}( R^\eps)}, \quad \forall \varphi \in W^{1,p}(R^\eps).
\end{align*}
It is very common to consider this kind of norms in works involving thin domains, see e.g. \cite{HR,R,PriRi}.

 Then, assuming that
 \begin{equation}\label{limit f12}
\frac{1}{\eps^{\gamma+1}}\int_{-\eps k^1_\eps(x)}^{\eps k^2_\eps(x)} \chi_{\theta^\epsilon}(x,y)f^\eps(x,y)dy \weto f_0(x) \quad \hbox{w}-L^2(I),
\end{equation}
 the main result obtained  is the following:
 
 \begin{theorem}\label{main appen}
Let $w^\eps$ be the solution of problem (\ref{1OPI0}). Given the function $P$ introduced in hypothesis {\bf (H.4)}, the definition of $f_0$ as provided in \eqref{limit f12}, and denoting $\hat{f}=\displaystyle\frac{f_0}{K^1 + K^2}$,  then we have 
$$
 w^\eps \to \hat w, \hbox{ w}-H^1(I),
$$
$$|||w^\eps - \hat w|||_{L^2(R^\eps)} \to 0,$$
 where $\hat w \in H^1(I)$ is the weak solution of  the following Neumann problem
\begin{equation*}\label{homogenized problem1}
\left\{
\begin{gathered}
-\frac{1}{P(K^1+K^2)}  {\hat w}_{xx} + \hat w = \hat f, \quad x \in I, \\
\hat w_x(a) = \hat w_x(b) = 0.
\end{gathered}
\right.
\end{equation*}
\end{theorem}

 Finally, we will get the boundary limiting problem for interesting particular cases such as quasi-periodicity or almost periodic framework.

\section{Reduction to an one-dimensional problem with oscillating coefficients}

In this section, we focus on the study of the elliptic problem \eqref{1OPI0} and begin by performing a change of variables that simplifies the domain in which the original problem is posed. As a matter of fact, we will be able to reduce the study of \eqref{1OPI0} in the thin domain $R^\eps$ to the study of an elliptic problem with oscillating coefficients capturing the geometry and oscillatory behavior of the open sets where the concentrations take place in the lower dimensional fixed domain $I$. This dimension reduction will be the key point to obtain the correct limiting equation. 
In order to state the main result of this section, let us first make some definitions. We will denote by 

\begin{equation} \label{def-eta}
\displaystyle \eta^i(\eps)=\sup_{x \in I}\Big|\eps \frac{d }{dx}k^i_\eps(x) \Big|>0,\quad i=1,2\quad \hbox{ and }\quad  \eta(\eps)=\eta^1(\eps)+\eta^2(\eps).
\end{equation}
Observe that from hypothesis {\bf (H.1)} we have $\eta(\eps)\eto 0$

Also, we denote by $K_\eps(x)=k^1_\eps(x)+k_\eps^2(x)$ (that is $\eps K_\eps(x)$ is the thickness of the thin domain $R^\eps$ at the point $x\in I$), so from (\ref{boundk}) there exist constants $K_0, K_1>0$ independent of $\eps>0$ such that
\begin{equation}\label{est K}
 0<K_0 \leq K_\eps(x) \leq K_1
 \,\,\mbox{for all}\,\, I \subset \R.  
\end{equation}

Now we will consider the following one dimensional problem

\begin{equation}\label{transformed-problem-hat}
\left\{
\begin{gathered}
- \frac{1}{K_\eps(x)}(K_\eps(x)  \hat{w}^\eps_x)_x + \hat w^\eps = \frac{1}{\eps^\gamma}\hat f^\eps
\quad \textrm{ in } I, \\
\hat{w}^\eps_x(a)=\hat{w}^\eps_x(b)=0
\end{gathered}
\right. 
\end{equation}
where 
\begin{equation}\label{fepsgorro}
\displaystyle \hat f^\eps(x)=\frac{1}{\eps K_\eps(x)}\int_{-\eps k^1_\eps(x)}^{\eps k^2_\eps(x)} \chi_{\theta^\epsilon}(x,y)f^\eps(x,y)dy.
\end{equation}

The key result of this section is the following

\begin{theorem}\label{main-reduction}
Let $w^\eps$ and $\hat w^\eps$ be the solutions of problems (\ref{1OPI0}) and (\ref{transformed-problem-hat}) respectively. Then, we have 
\begin{equation*}\label{main-estimate}
|||w^\eps-\hat w^\eps|||_{H^1(R^\eps)} \eto 0.
\end{equation*}
\end{theorem} 

\par\bigskip

In order to prove this result, we will need to obtain first some preliminary lemmas.  
We start transforming equation \eqref{1OPI0} into an equation in the modified thin domain
\begin{equation}\label{thin-intro1 appen}
R_a^\epsilon = \Big\{ (\bar x,\bar y) \in \R^{2} \; | \; \bar x \in I,  \; 0 < \bar y < \eps \,K_\eps(\bar x) \Big\}.
 \end{equation}
(see Figure \ref{thin2 appen}) 
where it can be seen that we have transformed the oscillations of both boundaries into oscillations of just the boundary at the top.  For this, we considering  the following family  of diffeomorphisms
\begin{equation*}
\begin{array}{rl}
L^\eps: R_a^\epsilon &\longrightarrow R^\eps \\
 (\bar x, \bar y)&\longrightarrow (x, y) := (\bar x, \,\bar y\, -\, \eps \,k^1_\eps(\bar x) ).
 \end{array}
\end{equation*}
Notice that the inverse of this diffeomorphism is  $(L^\eps)^{-1}(x,y)=(x,\, y \,+ \,\eps\, k^1_\eps(x))$. Moreover, from the structure of these diffeomorphisms and hypothesis {\bf (H.1)} we easily get that there exists a constant $C$ such that the Jacobian Matrix of $L^\eps$ and $(L^\eps)^{-1}$ satisfy 
\begin{equation}\label{bound-jacobian}
\|JL^\eps\|_{L^\infty},\|J(L^\eps)^{-1}\|_{L^\infty}\leq C.
\end{equation}
Moreover, we also have  $det(JL^\eps)(\bar x, \bar y)=det(J(L^\eps)^{-1})( x, y)=1$. 

We will show that the study of the limit behavior of the solutions of \eqref{1OPI0} is equivalent to analyze the behavior of the solutions of the following problem
\begin{equation} \label{OPI01 appen}
\left\{
\begin{gathered}
- \Delta v^\epsilon + v^\epsilon = \frac{1}{\eps^\gamma}\chi_{\theta^\epsilon_1}f_1^\eps
\quad \textrm{ in } R_a^\epsilon, \\
\frac{\partial v^\epsilon}{\partial \nu^\epsilon} = 0
\quad \textrm{ on } \partial R_a^\epsilon,
\end{gathered}
\right. 
\end{equation}
where the function $\chi_{\theta^\epsilon_1}: \R^{2} \rightarrow \R$ is the characteristic function of the narrow strip $\theta^\eps_1$ given by
\begin{equation}\label{theta1}
   \theta^\epsilon_1 = \Big\{ (\bar x, \bar y) \in \R^{2} \; | \;  \bar x \in I \subset \R,  \; \epsilon[ K_\eps(\bar x)  - \epsilon^\gamma H_\eps(\bar x)] < \bar y < \epsilon K_\eps(\bar x)  \Big\}. 
\end{equation}
The vector $\nu^\epsilon$ is the outward unit normal to $\partial R_a^\epsilon$  and 
\begin{equation}\label{def-f1eps}
f_1^\eps= f^\eps \circ L^\eps.
\end{equation}

\begin{figure}[H]
 \centering
   \includegraphics [width=8cm, height=3.5cm]{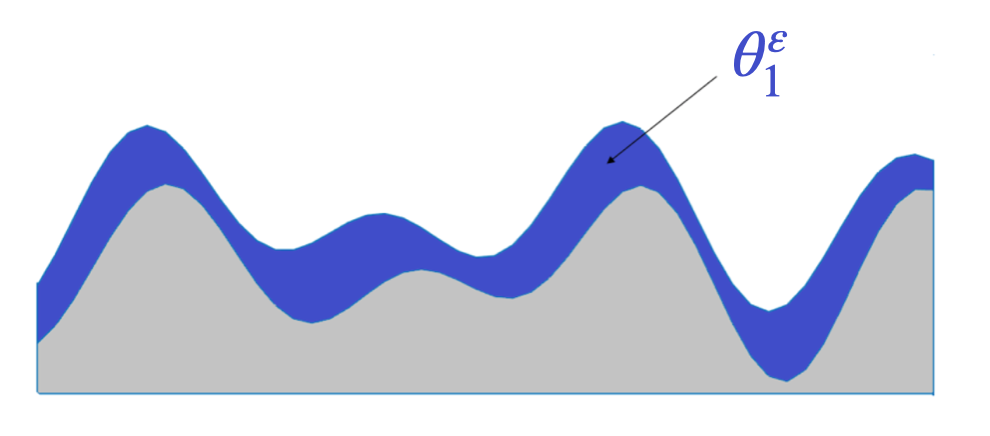}
   \caption{Thin domain $R_a^\epsilon$ obtained from $R^\eps$. }
   \label{thin2 appen}
\end{figure} 

\begin{remark}
Notice that 
\begin{equation}\label{a priori f_1}
\|f_1^\eps\|_{L^2(\theta^\epsilon_1)}^2=\int_{\theta^\epsilon_1}|f^\eps\circ L^\eps(\bar X)|^2d\bar X=\int_{\theta^\epsilon}|f^\eps(X)|^2|det(J(L^\eps)^{-1}))(X)|dX=\|f^\eps\|_{L^2(\theta^\epsilon)}^2,
\end{equation}
where $\bar X=(\bar x, \bar y)$ and $X=(x, y)$. Therefore, 
$|||f_1^\eps|||_{L^2(\theta^\epsilon_1)} = |||f^\eps|||_{L^2(\theta^\epsilon)}$ for all $f^\eps\in L^2(R^\eps)$. 
\end{remark}

 To achieve a better understanding of the behavior from the nonlinear concentrated term, we will analize the concentrating integral. We are now interested in analyzing concentrated integrals on the narrow strips $\theta_1^\epsilon$, to this end we will adapt results from \cite{AAM1} at Section 3. 
The next Lemma is about concentrating integrals, some estimates will be given in different Lebesgue and Sobolev-Bochner spaces.

\begin{lemma}\label{conc int}
Let $R_a^\eps$ and $\theta^\eps_1$ defined by (\ref{thin-intro1 appen}) and (\ref{theta1}) respectively. If $v^\eps \in W^{1,p}(R_a^\eps)$, there exists $C>0$ independent of $\eps$ such that, for $1- \frac{1}{p} < s \leq 1$, 
\begin{equation}\label{est veps1}
\frac{1}{\eps^\gamma}\int_{\theta_1^\eps}|v^\eps|^q d\bar X \leq C\|v^\eps\|^q_{L^q(I;W^{s,p}(0,\eps K_\eps(\bar x)))},\, \,\, \forall q \geq 1. 
\end{equation}
In particular, if $p \geq 2$, we have
\begin{equation}\label{est vesp2}
\frac{1}{\eps^\gamma}\int_{\theta_1^\eps}|v^\eps|^p d\bar X  \leq C\|v^\eps\|^q_{W^{1,p}(R^\eps_a)} ,\,\, \forall q \leq p.
\end{equation}
\end{lemma}
\begin{proof}
The estimate (\ref{est veps1}) is true by changing the $H^s$ space by $W^{s,p}$, for $1 - \frac{1}{p} < s \leq 1$, in \cite[Theorem 3.7]{AAM1}, which implies that exists $C>0$ independent of $\eps >0$ such that 

\begin{equation*}
\frac{1}{\eps^\gamma}\int_{\theta_1^\eps}|v^\eps|^q d\bar X \leq C\|v^\eps\|^q_{L^q(I;W^{s,p}(0,\eps K_\eps(\bar x)))},\, \,\, \forall q \geq 1. 
\end{equation*}

Also it is not difficult to see that $W^{1,p}(R^\eps_a)$ is included in $L^p(I;W^{1,p}(0,\eps K_\eps(\bar x)))$ and, consequently, in $L^p(I;W^{s,p}(0,\eps K_\eps(\bar x)))$, with constant independent of $\eps>0$, analogously as \cite[Preposition 3.6]{AAM1}.

On the other hand, if $p>2$, we have that $\frac{p}{q}>1$ for all $q<p$, and then, if we define ${\bar K_\eps}(\bar x)=\eps K_\eps(\bar x)$,
\begin{align*}
 \|v^\eps\|^q_{L^q(I;W^{1,p}(0,\bar K_\eps(\bar x))} &=  \int_I\Big(\int_0^{\bar K_\eps(\bar x)}|v^\eps(\bar x, \bar y)|^p d\bar y\Big)^\frac{q}{p}d\bar x \\&+  \int_I\Big(\int_0^{\bar K_\eps(\bar x)}|\partial_{\bar y}v^\eps(\bar x, \bar y)|^p d\bar y\Big)^\frac{q}{p}d\bar x\\
 &\leq K_1^{\frac{p-q}{p}}\Big(\int_I\int_0^{\bar K_\eps(\bar x)}|v^\eps(\bar x, \bar y)|^p d\bar  y d\bar x\Big)^\frac{q}{p}\\& +  K_1^{\frac{p-q}{p}}\Big(\int_I\int_0^{\bar K_\eps(\bar x)}|\partial_ {\bar y}v^\eps(\bar x, \bar y)|^p d\bar  y d\bar x\Big)^\frac{q}{p}\\
&\leq C \|v^\eps\|^q_{W^{1,p}(R^\eps_a)}
\end{align*}
\end{proof}

We would like to obtain a uniform boundedness for the solutions of the problem (\ref{OPI01 appen}) for any $\eps>0$ and $\gamma>0$, so from now on we will assume that  $f^\eps$ satisfies the following hypothesis
\begin{equation}\label{est f}
\frac{1}{\eps^\gamma}|||f^\eps|||^2_{L^2(\theta^\eps)} \leq c,
\end{equation}
for some positive constant $c$ independent of $\eps>0$.

\begin{lemma}\label{est v}
Consider the variational formulation of (\ref{OPI01 appen}). Then, there exists $c_0>0$, also independent of $\eps$, such that 
\begin{equation*}\label{a priori1 appen}
|||v^\eps|||_{H^1(R^\eps_a)}\leq c_0.
\end{equation*}
\end{lemma}
\begin{proof}
Taking the test function $v^\eps$ in the variational formulation of (\ref{OPI01 appen}) we have   
\begin{align*}|||v^\eps|||^2_{H^1(R^\eps_a)}  &= \frac{1}{\eps}\|v^\eps\|^2_{H^1(R^\eps_a)} \leq \Big(\frac{1}{\eps^{\gamma + 1}}\int_{\theta^\eps_1 }|f^\eps_1(\bar x, \bar y)|^2d\bar X\Big)^\frac{1}{2}\Big(\frac{1}{\eps^{\gamma + 1}}\int_{\theta^\eps_1 }|v^\eps(\bar x,\bar y)|^2 d\bar X\Big)^\frac{1}{2}\\
&\leq C\Big(\frac{1}{\eps^\gamma}|||f_1^\eps|||_{L^2(\theta^\eps_1)}^2\Big)^\frac{1}{2}|||v^\eps|||_{H^1(R^\eps_a)}
\end{align*}
using (\ref{a priori f_1}), (\ref{est f}) and Lemma \ref{conc int} we obtain
$$|||v^\eps|||_{H^1(R^\eps_a)}\leq c_0,$$ where $c_0$ is a positive constant independent of $\eps>0$.
\end{proof}

\par\medskip

The next result discusses the relationship between the solutions of problems \eqref{1OPI0} and \eqref{OPI01 appen}.  

\begin{lemma}\label{1 transform}
Let $w^\eps$ and $v^\eps$ be the solutions of problems \eqref{1OPI0} and \eqref{OPI01 appen} respectively, assuming (\ref{est f}), we have
\begin{equation*}\label{eq-Lemma}
|||w^\eps\circ L^\eps - v^\eps|||_{H^1(R_a^\epsilon)} \eto 0.
\end{equation*}

\end{lemma}

\begin{proof}
From the definition of $L^\eps$ we have
\begin{align*}
& \frac{\partial (w^\epsilon\circ L^\eps)}{\partial \bar x} =  \frac{\partial w^\epsilon}{\partial x}  - \eps\Big(\frac{d }{dx}k^1_\eps(x)\Big) \frac{\partial w^\epsilon}{\partial y},\\
& \frac{\partial (w^\epsilon\circ L^\eps)}{\partial \bar y} =  \frac{\partial w^\epsilon}{\partial y}.
 \end{align*}
In the new system of variables ($x=\bar x$ and $\bar y = y \,+\, \eps \,k^1_\eps(x)$) the variational formulation of \eqref{1OPI0} is given by
\begin{equation}\label{VFP0 appen}
\begin{split}
&\int_{R_a^\epsilon} \Big\{ \frac{\partial (w^\epsilon\circ L^\eps)}{\partial x} \frac{\partial \varphi}{\partial x} + \frac{\partial(w^\epsilon\circ L^\eps)}{\partial \bar y}\frac{\partial \varphi}{\partial \bar y} + (w^\epsilon\circ L^\eps)\varphi \Big\} d \bar X\\
&\quad +\int_{R_a^\epsilon}  \Big\{\eps \Big(\frac{d }{dx}k^1_\eps(x)\Big)\Big(\frac{\partial(w^\epsilon \circ L^\eps)}{\partial \bar y} \frac{\partial \varphi}{\partial x} + \frac{\partial (w^\epsilon\circ L^\eps)}{\partial x}\frac{\partial \varphi}{\partial \bar y}\Big)\Big\}d \bar X\\
& \quad+ \int_{R_a^\epsilon} \Big\{  \Big(\eps\frac{d}{dx}k^1_\eps(x)\Big)^2\frac{\partial(w^\epsilon \circ L^\eps)}{\partial \bar y} \frac{\partial \varphi}{\partial \bar y} \Big\}  d \bar X\\
&\quad = \frac{1}{\eps^\gamma}\int_{\theta^\eps_1} f^\eps_1 \varphi\, d \bar X, 
\quad \forall \varphi \in H^1(R_a^\epsilon).
\end{split}
\end{equation}
On the other hand, the weak formulation of \eqref{OPI01 appen} is: find $v^\epsilon \in H^1(R_a^\epsilon)$ such that 
\begin{equation} \label{VFP1 appen}
\int_{R_a^\epsilon} \Big\{ \frac{\partial v^\epsilon}{\partial x} \frac{\partial \varphi}{\partial x} 
+ \frac{\partial v^\epsilon}{\partial \bar y} \frac{\partial \varphi}{\partial \bar y}
+ v^\epsilon \varphi \Big\} \, d\bar X = \frac{1}{\eps^\gamma}\int_{\theta^\epsilon_1} f^\eps_1 \varphi \, d\bar X, 
\quad \forall \varphi \in H^1(R_a^\epsilon).
\end{equation}
Therefore, subtracting \eqref{VFP1 appen} from \eqref{VFP0 appen}, taking $(w^\epsilon\circ L^\eps - v^\eps)$ as a test function and after some computations and simplifications, we obtain

\begin{align*}
\|w^\eps\circ L^\eps-v^\eps\|_{H^1(R_a^\eps)}^2
\leq c\,\eta^1(\eps)\|\nabla (w^\eps\circ L^\eps)\|_{L^2(R_a^\eps)}\|\nabla (w^\eps\circ L^\eps-v^\eps)\|_{L^2(R_a^\eps)},
\end{align*}
where $c>0$ is a constant independent of $\eps$ and $\eta^1(\eps)$ is given by (\ref{def-eta}).

This implies 
\begin{align*}
\|w^\eps\circ L^\eps-v^\eps\|_{H^1(R_a^\eps)}
&\leq c\,\eta^1(\eps)\|w^\eps\circ L^\eps\|_{H^1(R_a^\eps)}\\
&\leq c\,\eta^1(\eps)\|w^\eps\circ L^\eps-v^\eps\|_{H^1(R_a^\eps)}+ c\,\eta^1(\eps)\|v^\eps\|_{H^1(R_a^\eps)},
\end{align*}
and therefore using Lemma \ref{est v} since $v_\eps$ is the solution of \eqref{OPI01 appen} we have
\begin{align*}
|||w^\eps\circ L^\eps-v^\eps|||_{H^1(R_a^\eps)}
\leq \frac{c\,\eta^1(\eps)}{1-c\,\eta^1(\eps)}|||v^\eps|||_{H^1(R_a^\eps)}\leq \frac{c\,\eta^1(\eps)}{1-c\,\eta^1(\eps)}c_0,
\end{align*}
the result follows from the definition of $\eta^1(\eps)$. 
\end{proof}


\begin{remark}
Notice that from the point of view of the limit behavior of the solutions it is the same to study problem \eqref{1OPI0} defined in the doubly oscillating thin domain as to analyze problem \eqref{OPI01 appen} posed in a thin domain with just one oscillating boundary. It is important to note that this is true because of {\bf(H.1)}. If that assumption is not satisfied, at least for $k_\eps^1$, then the simplification it is not possible. For instance, if we have a domain with oscillations with the same period in both boundaries like this one
$$R^\epsilon = \Big\{ (x, y) \in \R^2 \; | \;  x \in (0,1),  \; -\eps(2 - g(x/\eps)) < y < \epsilon g(x/\eps)  \Big\},$$
where  $g: \R \to \R $ is a smooth  $L-$periodic function.
Observe that in this particular case we have
$$k^1_\eps(x)=2 - g(x/\eps), \quad k^2_\eps(x)=g(x/\eps).$$
Therefore, it follows straightforward that condition \eqref{main} is not satisfied.
Observe that the original problem \eqref{1OPI0} for this particular thin domain is in the framework of the classical periodic homogenization while the converted problem \eqref{OPI01 appen} is posed in a rectangle of height $\eps$ where homogenization theory is not necessary to analyze the behavior of the solutions. Indeed, if $k_\eps^1$ does not satisfy \eqref{main} the solutions of problems \eqref{1OPI0} and \eqref{OPI01 appen} are not comparable in general.
 \end{remark}
 
 Now we define a transformation on the thin domain $R^\eps_a$, which will map $R^\eps_a$ into the fixed rectangle
$Q=I \times (0,1)$. This transformation is given by
\begin{equation*}
\begin{array}{rl}
S^\eps: Q &\longrightarrow R_a^\epsilon \\
 (x, y)&\longrightarrow (\bar x,  \bar y) := (x\,, \,y\,\eps K_\eps(x)).
 \end{array}
 \end{equation*}
We recall that $K_\eps(x)=  k_\eps^2(x) + k_\eps^1(x)$.  

Using the chain rule and standard computations it is not difficult to see that  there exist $c,C>0$  such that if $v^\eps\in H^1(R_a^\eps)$ and $u^\eps=v^\eps \circ S^\eps\in H^1(Q)$ then  the following estimates hold

\begin{equation}\label{comparison-L2}
c|||v^\eps|||^2_{L^2(R^\eps_a)}\leq \|u^\eps\|^2_{L^2(Q)}\leq C|||v^\eps|||^2_{L^2(R^\eps_a)},
\end{equation}

\begin{equation}\label{comparison-Dy}
c\Big|\Big|\Big|\frac{\partial v^\eps}{\partial \bar y}\Big|\Big|\Big|^2_{L^2(R^\eps_a)}\leq \frac{1}{\eps^2}\Big\|\frac{\partial u^\eps}{\partial  y}\Big\|^2_{L^2(Q)}\leq C\Big|\Big|\Big|\frac{\partial v^\eps}{\partial \bar y}\Big|\Big|\Big|^2_{L^2(R^\eps_a)},
\end{equation}

\begin{equation}\label{comparison-H1}
c|||\nabla v^\eps|||^2_{L^2(R^\eps_a)}\leq \|\nabla u^\eps\|^2_{L^2(Q)}\leq C|||\nabla v^\eps|||^2_{L^2(R^\eps_a)}.
\end{equation}

\par\medskip

Now, under this change of variables and defining $u^\eps=v^\eps\circ S^\eps$ where $v^\eps$ satisfies \eqref{OPI01 appen} and $f^\eps_2= f^\eps_1 \circ S^\eps$, where $f_1^\eps$ is defined in \eqref{def-f1eps}, problem \eqref{OPI01 appen} becomes
\begin{equation}\label{transformed-problem}
\left\{
\begin{gathered}
- \frac{1}{K_\eps}\hbox{div}\big(B^\eps(u^\eps)\big)  + u^\epsilon = \frac{1}{\eps^\gamma}\chi_{\theta^\epsilon_2}f_2^\eps
\quad \textrm{ in } Q, \\
B(u^\eps)\cdot \eta  = 0
\quad \textrm{ on } \partial Q,\\
u^\eps=v^\eps \circ S^\eps \textrm{ in } Q,
\end{gathered}
\right. 
\end{equation}
where the function $\chi_{\theta^\epsilon_2}: \R^2 \rightarrow \R$ is the characteristic function of the narrow strip $\theta^\eps_2$ given by

\begin{equation}\label{theta2}
   \theta^\epsilon_2 = \Big\{ (x, y) \in \R^2 \; | \;  x \in I ,  \; 1  - \epsilon^\gamma \frac{H_\eps(x)}{K_\eps(x)}< y < 1 \Big\}.
\end{equation}
The vector $\eta$ denotes the outward unit normal vector field to $\partial Q$ and 
\begin{equation*}
B^\eps(u^\eps) = \Big( K_\eps \frac{ \partial u^\eps}{\partial x} - y \displaystyle\frac{dK_\eps}{dx}\frac{ \partial u^\eps}{\partial y}\,,\, - y \displaystyle\frac{dK_\eps}{dx}\frac{ \partial u^\eps}{\partial x} + \Big(\frac{y^2}{ K_\eps}\Big(\displaystyle\frac{dK_\eps}{dx}\Big)^2 +
\frac{1}{\eps^2 K_\eps}\Big)\frac{ \partial u^\eps}{\partial y}\Big).
\end{equation*}

Notice that 
\begin{equation}\label{a priori f_2}
c|||f^\eps_1|||^2_{L^2(\theta^\eps_1)}\leq \|f^\eps_2\|^2_{L^2(\theta_2^\eps)}\leq C|||f^\eps_1|||^2_{L^2(\theta^\eps_1)}.  
\end{equation}
 Therefore from (\ref{a priori f_1}) and (\ref{a priori f_2}) we have
\begin{equation}\label{a priori f}
c|||f^\eps|||^2_{L^2(\theta^\eps)}\leq \|f^\eps_2\|^2_{L^2(\theta_2^\eps)}\leq C|||f^\eps|||^2_{L^2(\theta^\eps)}.
\end{equation}

In the new system of coordinates we obtain a domain which is neither thin nor oscillating anymore. In some sense, we have rescaled the neighborhood (\ref{theta1}) into the strip $\theta_2^\eps \subset Q$ and substituted the thin domain $R_a^\eps$ for a domain $Q$ independent on $\eps$, at a cost of replacing the oscillating thin domain by
oscillating coefficients in the differential operator.

In order to analyze the limit behavior of the solutions of  \eqref{transformed-problem} we establish the relation to the solutions of the following easier problem
\begin{equation}\label{transformed-problem1}
\left\{
\begin{gathered}
- \frac{1}{K_\eps}\Big[\frac{\partial}{\partial x}\Big(K_\eps \frac{\partial w_1^\eps}{\partial x}\Big) + \frac{1}{\eps^2 K_\eps} \frac{\partial^2w_1^\eps}{\partial y^2}\Big]  + w_1^\epsilon = \frac{1}{\eps^\gamma}\chi_{\theta^\epsilon_2}f_2^\eps
\quad \textrm{ in } Q, \\
K_\eps \frac{\partial w_1^\eps}{\partial x}\, \eta_1 + \frac{1}{\eps^2 K_\eps} \frac{\partial w_1^\eps}{\partial y}\, \eta_2 = 0  \quad \textrm{ on } \partial Q,
\end{gathered}
\right. 
\end{equation}
where $\eta = (\eta_1, \eta_2)$ is the outward unit normal to $\partial Q$.

Before obtaining a priori estimates for the solutions of (\ref{transformed-problem1}), it is necessary to prove the following result

\begin{lemma}\label{conc int w}
Let $Q=I \times (0,1)$ and $\theta_2^\eps$ like defined in (\ref{theta2}), suppose that $w_1^\eps \in H^s(Q)$ with $\frac{1}{2} <s \leq \ 1$ and $s- 1 \geq - \frac{1}{q}$. Then, for small $\eps_0 >0$, there exist a constant $C > 0$ independent of $\eps$ and $w_1^\eps$, such that for any $0 < \eps \leq \eps_0$, we have 
$$\frac{1}{\eps^\gamma}\int_{\theta_2^\eps}|w_1^\eps|^q dxdy \leq C\|w_1^\eps\|^q_{H^s(Q)}.$$
\end{lemma}
\begin{proof}
First we observe that 
$$\frac{1}{\eps^\gamma}\int_{\theta_2^\eps}|w_1^\eps(x,y)|^q dxdy = \frac{1}{\eps^\gamma}\int_I\int_0^{\eps^\gamma\frac{H_\eps(x)}{K_\eps(x)}}|w_1^\eps(x,1-y)|^q dydx \leq \int_{r_\eps}|w_1^\eps(x,1 -y)|^q dxdy,$$
where $r_\eps$ is the strip without oscillatory behavior given by 
$$r_\eps=\Big\{(x,y) \in \R^2 \,|\, x \in I , 0 < y < \eps^\gamma \frac{H_1}{K_0} \Big\}.$$
Using \cite[Lemma 2.1]{AAA} we have that there exists $\eps_0>0$ and $C>0$ independent of $\eps$ and $v^\eps=w^\eps_1 \circ \tau$ such that 
$$\frac{1}{\eps^\gamma}\int_{\theta_2^\eps}|w_1^\eps(x,y)|^q dxdy \leq \|v^\eps\|^q_{H^s(\tau^{-1}(Q))}, \,\,\,\forall \eps \in (0,\eps_0),$$
where we are taking $\tau : \R^2 \rightarrow \R^2$ given by $\tau(x,y)=(x,1-y)$. From \cite[Section 2]{AAA} we have that the norms $\|v^\eps\|_{H^s(\tau^{-1}(Q))}$ and $\|w_1^\eps\|_{H^s(Q)}$ are equivalents, so we can conclude the proof.
\end{proof}

Observe that assuming (\ref{est f}), estimates (\ref{a priori f}), under the assumptions on the functions $k^1_\eps$ and $k^2_\eps$, Lemma \ref{conc int w}, so the equation \eqref{transformed-problem1} admits a unique solution $w_1^\eps \in H^1(Q)$, which satisfies the priori estimates
\begin{equation} \label{EST1 appen}
\begin{gathered}
\| w_1^\epsilon \|_{L^2(Q)}, \, \, \, \Big\| \frac{\partial w_1^\epsilon}{\partial x} \Big\|_{L^2(Q)}, \, \, \, \frac{1}{\epsilon} \Big\| \frac{\partial w_1^\epsilon}{\partial y} \Big\|_{L^2(Q)} 
\le C, 
\end{gathered}
\end{equation}
where the positive constant $C$ is independet of $\eps>0$.

\begin{lemma}\label{2 transform}
Let $u^\eps$ and $w_1^\eps$ be the solution of problems \eqref{transformed-problem} and \eqref{transformed-problem1} respectively and suppose that $f^\eps$ satisfying (\ref{est f}). Then, we have
\begin{align*}
&\Big|\Big|\frac{\partial(u^\eps - w_1^\eps)}{\partial x}\Big|\Big|_{L^2(Q)}^2 + \frac{1}{\eps^2}\Big|\Big|\frac{\partial(u^\eps - w_1^\eps)}{\partial y}\Big|\Big|_{L^2(Q)}^2 +||u^\eps - w_1^\eps||_{L^2(Q)}^2 \eto 0.
 \end{align*}
\end{lemma}

\begin{proof}
Subtracting the weak formulation of \eqref{transformed-problem1} from the weak formulation of \eqref{transformed-problem} and choosing $u^\eps-w_1^\eps $ as test function we get
\begin{equation*}\label{21 transform}
\begin{split}
\int_{Q}&\Big\{ K_\eps\Big(\frac{\partial (u^\epsilon - w_1^\eps)}{\partial x}\Big)^2 + \frac{1}{\eps^2 K_\eps}\Big(\frac{\partial (u^\epsilon - w_1^\eps)}{\partial y}\Big)^2 + 
K_\eps\big(u^\epsilon - w_1^\eps\big)^2\Big\} dx dy\\
=& \int_{Q} \Big\{ y\frac{\partial u^\eps}{\partial y}\displaystyle\frac{dK_\eps}{dx}\frac{\partial (u^\eps-w^\eps_1 )}{\partial x} + y\displaystyle\frac{dK_\eps}{dx}\frac{\partial u^\eps}{\partial x}\frac{\partial (u^\eps-w_1^\eps )}{\partial y} - \frac{y^2}{K_\eps}\Big(\displaystyle\frac{dK_\eps}{dx}\Big)^2\frac{\partial u^\eps}{\partial y}\frac{\partial (u^\eps-w_1^\eps )}{\partial y}\Big\}dxdy\\
\leq& \,C\eta(\eps)\int_{Q} \Big\{ \frac{1}{\eps}\Big|\frac{\partial u^\eps}{\partial y}\Big| \Big(\Big|\frac{\partial u^\eps}{\partial x}\Big| + \Big|\frac{\partial w^\eps_1}{\partial x}\Big|\Big)+  \Big|\frac{\partial u^\eps}{\partial x}\Big|\frac{1}{\eps}\Big(\Big|\frac{\partial u^\eps }{\partial y}\Big| + \Big|\frac{\partial w_1^\eps }{\partial y}\Big| \Big)\Big\}dxdy\\
 +&\, C_0 (\eta(\eps))^2 \int_{Q}\Big(\frac{1}{\eps^2}\Big|\frac{\partial u^\eps}{\partial y}\Big|^2 + \frac{1}{\eps}\Big|\frac{\partial u^\eps}{\partial y}\Big|\frac{1}{\eps}\Big|\frac{\partial w_1^\eps}{\partial y}\Big|\Big)dxdy ,
\end{split}
\end{equation*}
where $C$ and $C_0$ are positive constants which does not depend on $\eps$.
Taking into account that  $\eps \Big|\displaystyle\frac{dK_\eps}{dx}\Big| \leq \eta(\eps)$, convergence (\ref{main}), estimates \eqref{comparison-Dy} and\eqref{comparison-H1} applied to $u^\eps$ and $v^\eps$, Lemma \ref{est v}, the priori estimate of $w_1^\eps$ (see  \eqref{EST1 appen}) and following standard computations 
we obtain the result.
\end{proof}

Finally, we compare the behavior of the solution of \eqref{transformed-problem1} to the solutions of the following problem posed in $I \subset \R$
\begin{equation}\label{transformed-problem3}
\left\{
\begin{gathered}
- \frac{1}{K_\eps}\frac{\partial}{\partial x}\Big(K_\eps \frac{\partial u_1^\eps}{\partial x}\Big)+ u_1^\epsilon = \frac{1}{\eps^\gamma} f^\eps_3
\quad \textrm{ in } I, \\
(u_1^\eps)_x(a)= (u_1^\eps)_x(b) = 0  \quad \textrm{ on } \partial I,
\end{gathered}
\right. 
\end{equation}
where $f^\eps_3(x)=\displaystyle\int_{1-\eps^\gamma\frac{H_\eps(x)}{K_\eps(x)}}^1 f_2^\eps(x,y) dy$ for a.e. $x \in I$, which is a function depending only on the $x$ variable.

Then, considering $u_1^\eps(x)$ as a function defined in $Q$ (extending it in a constant way in the $y$ direction) we prove the following lemma.
\begin{lemma}\label{3 transform}
Let $u^\eps_1$ and $w_1^\eps$ be the solution of problems \eqref{transformed-problem3} and \eqref{transformed-problem1} respectively and suppose that $f^\eps$ satisfying (\ref{est f}).  Then, we have
$$\Big|\Big|\frac{\partial(w^\eps_1 - u_1^\eps)}{\partial x}\Big|\Big|_{L^2(Q)}^2 + \frac{1}{\eps^2}\Big|\Big|\frac{\partial w_1^\eps}{\partial y}\Big|\Big|_{L^2(Q)}^2 +||w^\eps_1 - u_1^\eps||^2_{L^2(Q)} \eto 0.$$
\end{lemma}
\begin{proof}
Taking $w_1^\eps - u_1^\eps$  as test function in the variational formulation of 
\eqref{transformed-problem1} and $\int_0^1 w_1^\eps dy - u_1^\eps$ as test function in the variational formulation of \eqref{transformed-problem3}
and subtracting both weak formulations we obtain
\begin{equation} \label{31 transform}
\begin{split}
\int_{Q}&\Big\{ K_\eps\Big(\frac{\partial (w^\epsilon_1 - u_1^\eps)}{\partial x}\Big)^2 + \frac{1}{\eps^2 K_\eps}\Big(\frac{\partial w_1^\eps}{\partial y}\Big)^2 + 
K_\eps \big(w_1^\epsilon - u_1^\eps\big)^2\Big\} dx dy\\
&= \frac{1}{\eps^\gamma}\int_{Q} K_\eps \chi_{\theta^\epsilon_2}f_2^\eps(w_1^\eps - u_1^\eps) dxdy - \frac{1}{\eps^\gamma}\int_{Q}K_\eps f_3^\eps( w_1^\eps - u_1^\eps) dxdy.
\end{split}
\end{equation}
Now we analyze the two terms in the right hand-side. First, taking into account the definition of $f_3^\eps$ we have that for any function $\varphi$ defined in $I$, that is, $\varphi=\varphi(x)$ 
$$\int_{Q} K_\eps (f_3^\eps - \chi_{\theta^\epsilon_2}f_2^\eps) \varphi dxdy=  \int_{I} \varphi K_\eps f_3^\eps dx - \int_{\theta^\eps_2}\varphi K_\eps f_2^\eps dy   = 0.$$
In particular  $\displaystyle\int_{Q} K_\eps (f_3^\eps - \chi_{\theta^\epsilon_2}f_2^\eps) u_1^\eps dxdy=0$
and $\displaystyle\int_{Q} K_\eps (f_3^\eps - \chi_{\theta^\epsilon_2}f_2^\eps) w_1^\eps(x,0)dxdy=0$. Hence, using Holder inequality, (\ref{est H}), (\ref{est K}), (\ref{est f}), (\ref{a priori f}) and \eqref{EST1 appen} we get
\begin{align*}
&\frac{1}{\eps^\gamma}\Big|\int_{Q} K_\eps (\chi_{\theta^\epsilon_2}f_2^\eps - f_3^\eps) w_1^\eps dxdy \Big| = \frac{1}{\eps^\gamma}\Big|\int_{Q} K_\eps (\chi_{\theta^\epsilon_2}f_2^\eps - f_3^\eps) (w_1^\eps (x,y)- w_1^\eps(x,0) ) dxdy\Big|\\
&\leq \frac{K_1}{\eps^\gamma}\int_{Q} \Big(\Big|\chi_{\theta^\epsilon_2}f_2^\eps - f_3^\eps\Big| \Big| \int_0^y \frac{\partial w_1^\eps}{\partial s}(x,s) ds \Big|\Big) dxdy\\
&\leq \frac{K_1}{\eps^\gamma}\Big[\int_{\theta_2^\eps} |f_2^\eps| \Big| \Big(\int_0^1 \Big|\frac{\partial w_1^\eps}{\partial s}(x,s)\Big| ds\Big) dxdy + \int_{I} |f_3^\eps| \Big(\int_0^1 \Big|\frac{\partial w_1^\eps}{\partial s}(x,s)\Big| ds\Big) dx \Big]\\
& \leq \frac{K_1}{\eps^\gamma}\Big[\|f_2^\eps\|_{L^2(\theta_2^\eps)} \Big\|\int_0^1 \Big|\frac{\partial w_1^\eps}{\partial s}(x,s)\Big| ds\Big\|_{L^2(\theta_2^\eps)}+ \int_{I} \Big|\displaystyle\int_{1-\eps^\gamma\frac{H_\eps(x)}{K_\eps(x)}}^1 f_2^\eps(x,y) dy\Big|\Big(\int_0^1 \Big|\frac{\partial w_1^\eps}{\partial s}(x,s)\Big| ds\Big) dx \Big]\\
& \leq \frac{K_1}{\eps^\gamma}\Big[\|f_2^\eps\|_{L^2(\theta_2^\eps)} \Big(\eps^\gamma \frac{H_1}{K_0}\Big)^\frac{1}{2}\Big\| \frac{\partial w_1^\eps}{\partial y}\Big\|_{L^2(Q)}+ \int_{I} \displaystyle\int_{1-\eps^\gamma\frac{H_\eps(x)}{K_\eps(x)}}^1 |f_2^\eps| \Big(\int_0^1 \Big|\frac{\partial w_1^\eps}{\partial s}(x,s)\Big| ds\Big) dydx \Big]\\
&\leq \frac{K_1}{\eps^\gamma}\Big[\|f_2^\eps\|_{L^2(\theta_2^\eps)} \Big(\eps^\gamma \frac{H_1}{K_0}\Big)^\frac{1}{2}\Big\| \frac{\partial w_1^\eps}{\partial y}\Big\|_{L^2(Q)}+ \displaystyle\int_{\theta_2^\eps} |f_2^\eps| \Big(\int_0^1 \Big|\frac{\partial w_1^\eps}{\partial s}(x,s)\Big| ds\Big) dxdy \Big]\\
& \leq \frac{K_1}{\eps^\gamma}\Big[\|f_2^\eps\|_{L^2(\theta_2^\eps)} \Big(\eps^\gamma \frac{H_1}{K_0}\Big)^\frac{1}{2}\Big\| \frac{\partial w_1^\eps}{\partial y}\Big\|_{L^2(Q)}+  \|f_2^\eps\|_{L^2(\theta_2^\eps)} \Big\|\int_0^1 \Big|\frac{\partial w_1^\eps}{\partial s}(x,s)\Big| ds\Big\|_{L^2(\theta_2^\eps)}\Big]\\
&\leq C\eps^{-\frac{\gamma}{2}}\|f_2^\eps\|_{L^2(\theta_2^\eps)}\Big\| \frac{\partial w_1^\eps}{\partial y}\Big\|_{L^2(Q)} \leq
C_0\eps
\end{align*}
whrere $C_0$ is a positive constant wich does not depend on $\eps$. Therefore, from \eqref{31 transform} the lemma is proved.
\end{proof}

\begin{remark}
Notice that $u^\eps_1$, the solution of \eqref{transformed-problem3} coincides with $\hat w$, the solution of \eqref{transformed-problem-hat}. Indeed, by the definition of $f^\eps_3$, $f^\eps_2$ and $f^\eps_1$, we get 
\begin{align*}
 f^\eps_3(x)&=\int_{1-\eps^\gamma\frac{H_\eps(x)}{K_\eps(x)}}^1 f_1^\eps\circ S^\eps(x,y) dy=\frac{1}{\eps K_\eps(x)}\int_{\eps [K_\eps(x) -\eps^{\gamma}H_\eps(x)]}^{\eps K_\eps(x)} f^\eps_1(x,y)dy\\ 
 &=\frac{1}{\eps K_\eps(x)}\int_{\eps [K_\eps(x) -\eps^{\gamma}H_\eps(x)]}^{\eps K_\eps(x)} f^\eps \circ L^\eps(x,y)dy =
\frac{1}{\eps K_\eps(x)}\int_{\eps [k^2_\eps(x) - \eps^\gamma H_\eps(x)]}^{\eps k^2_\eps(x)} f^\eps(x,y)dy\\
&=\hat f^\eps(x)
\end{align*}
\end{remark}

After these lemmas we can provide a proof of the main result of this section.

\par\bigskip 

\par\noindent {\sl Proof of Therorem \ref{main-reduction}.} Let $w^\eps$ and $\hat w^\eps$ the solution of (\ref{1OPI0}) and (\ref{transformed-problem-hat}) respectively. Then, from \eqref{bound-jacobian} and the fact that $u_1^\eps$ does not depend on the $y$ variable, we get
\begin{align*}
&|||w^\eps-\hat w^\eps|||^2_{H^1(R^\eps)}=|||w^\eps- u_1^\eps|||^2_{H^1(R^\eps)}\\
&\leq C|||w^\eps\circ L^\eps - u_1^\eps\circ L^\eps |||^2_{H^1(R^\eps_a)}=C|||w^\eps\circ L^\eps - u_1^\eps |||^2_{H^1(R^\eps_a)}, 
\end{align*}
where $C$ is a positive constant independent of $\eps>0$. But,
\begin{align*}
&|||w^\eps\circ L^\eps - u_1^\eps |||^2_{H^1(R^\eps_a)}\leq 2|||w^\eps\circ L^\eps - v^\eps|||^2_{H^1(R^\eps_a)}+2|||v^\eps-u_1^\eps |||^2_{H^1(R^\eps_a)},
\end{align*}
where we have used the inequality $(a+b)^2\leq 2a^2+2b^2$. 

Now, using that $u^\eps=v^\eps\circ S^\eps$ and \eqref{comparison-L2}, \eqref{comparison-Dy}, \eqref{comparison-H1} and that $u_1^\eps$ is independent of $y$ and therefore $u_1^\eps\circ S^\eps=u_1^\eps$, we get 
$$|||v^\eps-u_1^\eps |||^2_{H^1(R^\eps_a)}\leq C \Big(\|u^\eps -u_1^\eps ||^2_{H^1(Q)}+\frac{1}{\eps^2}\Big\|\frac{\partial u^\eps}{\partial y}\Big\|^2_{L^2(Q)}\Big)$$
and applying now the triangular inequality
$$\leq C\Big(2 \|u^\eps-w_1^\eps\|^2_{H^1(Q)}+\frac{2}{\eps^2}\Big\|\frac{\partial (u^\eps-w_ 1^\eps)}{\partial y}\Big\|^2_{L^2(Q)}+ 2\|w_1^\eps-u_1^\eps\|^2_{H^1(Q)}+\frac{2}{\eps^2}\Big\|\frac{\partial w_1^\eps}{\partial y}\Big\|^2_{L^2(Q)}\Big).$$

Putting all this inequalities together and using Lemmas \ref{1 transform}, \ref{2 transform} and  \ref{3 transform} we prove the result. 


\section{Limit problem for the elliptic equation}

In view of Theorem \ref{main-reduction}, the homogenized limit problem of \eqref{1OPI0} will be obtained passing to the limit in the reduced problem \eqref{transformed-problem-hat}.  

As briefly explained in the introduction, we will explicitly derive the homogenized limit problem of \eqref{transformed-problem-hat} for oscillating functions satisfying {\bf (H)}. It is important to note that this situation encompasses the classical scenario where both the top and bottom boundaries are represented by the graphs of two periodic functions.

Recall that we assume that $\hat f^\eps$ defined in (\ref{fepsgorro}) satisfies the following convergence
\begin{equation}\label{limit f1}
\frac{1}{\eps^\gamma}K_\eps(\cdot) \hat f^\eps(\cdot) =\frac{1}{\eps^{\gamma+1}}\int_{-\eps k^1_\eps(x)}^{\eps k^2_\eps(x)} \chi_{\theta^\epsilon}(x,y)f^\eps(x,y)dy \weto f_0(\cdot) \quad \hbox{w-}L^2(I),
\end{equation}
 for certain $f_0\in L^2(I)$.

To begin, we establish a priori estimates of ${\hat w}^\eps$ that are independent of the specific functions $k^1_\eps(x)$ and $k^2_\eps(x)$. We leverage the fact that both oscillatory functions are uniformly bounded.

\begin{lemma}\label{est V}
Consider the variational formulation of (\ref{transformed-problem-hat}). Then, there exists $c_0>0$, independent of $\eps$, such that 
\begin{equation*}\label{a priori1 appen}
\|\hat w^\eps\|_{H^1(I)}\leq c_0.
\end{equation*}
\end{lemma}
\begin{proof}
The weak formulation of \eqref{transformed-problem-hat} is given by 
\begin{equation}\label{weak-version-final-problem1}
\int_I\Big\{ K_\eps \hat w^\eps_x \phi_x + K_\eps {\hat w}^\eps\phi \big\} \,  dx =\frac{1}{\eps^\gamma}\int_I K_\eps \hat f^\eps \phi \, dx, \quad \hbox{ for all } \phi\in H^1(I).
\end{equation}
Considering ${\hat w}^\eps$ as a test function in \eqref{weak-version-final-problem1} and using (\ref{boundk}), (\ref{est H}) and (\ref{est f}),   we get   
\begin{align*}\|\hat w^\eps\|^2_{H^1(I)}  & \leq 
C\frac{1}{\eps^{\gamma + 1}}\int_I\int_{-\eps k^1_\eps(x))}^{\epsilon \, k^2_\eps(x))}\chi_{\theta^\epsilon}(x,y) f^\eps (x, y) \hat w^\eps(x) dydx\\
&\leq  C\Big(\frac{1}{\eps^{\gamma + 1}}\int_{\theta^\eps }|f^\eps|^2dxdy\Big)^\frac{1}{2}\Big(\frac{1}{\eps^{\gamma + 1}}\int_{\theta^\eps }|\hat w^\eps|^2 dxdy\Big)^\frac{1}{2}\\
&\leq C\Big(\frac{1}{\eps^\gamma}|||f^\eps|||_{L^2(\theta^\eps)}^2\Big)^\frac{1}{2}\Big(\int_I  H_\eps|\hat w^\eps|^2 dx\Big)^\frac{1}{2}\\
&\leq C_1\|\hat w^\eps\|_{L^2(I)},
\end{align*}
where $C$ and $C_1$ are positive constants independents of $\eps>0$, so we obtain
$$\|\hat w^\eps\|_{H^1(I)}\leq c_0,$$ where $c_0$ is a positive constant independent of $\eps>0$.
\end{proof}

Thus, by weak compactness there exists $\hat{w} \in H^1(I)$ such that, up to subsequences
\begin{equation}\label{weak u1 appen}
 {\hat w}^\epsilon \weto \hat{w} \quad  \hbox{ w}-H^1(I).
 \end{equation}

Now, we are poised to derive the homogenized limit problem using the classical approach for homogenization of variable coefficients in one dimensional problems. Observe that  $K_\eps \hat w^\eps_x$ is uniformly bounded in $L^2(I)$ since 
$$||\hat w^\eps_x||_{L^2(I)}\leq C \quad  \hbox{ and } \quad 0<K_\eps(x)< C^1_2 +C^2_2, \hbox{ for each } x \in I.$$
Moreover, taking into account that $$
(K_\eps \hat w^\eps_x)_x= -\frac{1}{\eps^\gamma}K_\eps\hat f^\eps + K_\eps {\hat w}^\epsilon,
$$
we deduce that  $K_\eps \hat w^\eps_x$ is uniformly bounded in $H^1(I).$ Then, it follows that there exists a function $\sigma$ such that, up to subsequences,
$$K_\eps \hat w^\eps_x\longrightarrow \sigma \quad \hbox{strongly in } L^2(I).$$
Thus, using  {\bf(H.4)} we have
$$ \hat w^\eps_x = \frac{1}{K_\eps} \Big(K_\eps \hat w^\eps_x\Big)\,{\weto} \,P \sigma \quad  \hbox{ w}-L^2(I). $$

Consequently, due to convergence \eqref{weak u1 appen} we have
$$\hat{w}_x= P \sigma,$$
or equivalently,
\begin{equation}\label{product conv appen}
  K_\eps \hat w^\eps_x\eto \frac{1}{P} \hat{w}_x\quad \hbox{strongly in } L^2(I).
\end{equation}

Therefore, in view of \eqref{limit f1}, \eqref{weak u1 appen}  and \eqref{product conv appen} we can pass to the limit in (\ref{weak-version-final-problem1}) and we obtain the following weak formulation
$$\int_I \Big\{\frac{1}{P} \hat{w}_x  \phi_x + (K^1+K^2) \hat{w}\phi\Big\}\, dx = \int_I f_0\phi\, dx.$$

With this, Theorem \ref{main appen} is proven.

\section{Quasi-periodic and almost periodic setting}
In this section, we will analyze specific and interesting cases that satisfy the hypotheses {\bf (H)}. We will begin by examining the scenario where 
the oscillating boundary is given by quasi-periodic functions. In particular, we consider
\begin{equation}\label{quasiper}
 k^1_\eps(x)=h(x/\eps^\alpha), \quad k^2_\eps(x)=g(x/\eps^\beta),
 \end{equation}
where $0<\alpha, \beta<1$ and the functions $g,h \,: \R \to \R $ are $C^1$ quasi-periodic functions verifying 
\begin{align}
&0\leq h_0\leq h(\cdot)\leq h_1,\label{cota h}\\
&0<g_0 \leq g(\cdot)\leq g_1.\label{cota g}
\end{align}

Therefore $g,h \,: \R \to \R $ present a combination of multiple frequencies that are rationally independent. That is,  there exist two $C^1$ periodic functions $\Bar{h} \,: \R^n \to \R $ and $\Bar{g}\,:\R^m \to \R$ , $n,m\in \mathbb{N}$, such that
 \begin{equation}\label{quasi}
 h(x)=\Bar{h}(\underbrace{x,\cdots,x}_{n}),\quad
 g(x)=\Bar{g}(\underbrace{x,\cdots,x}_{m}).
 \end{equation}
 Notice that $\Bar{h}$ and $\Bar{g}$ are periodic with respect to each of their arguments. For example in the case of $\Bar{h}$, for all $1\leq j\leq n$, there exists $L_j>0$ such that $\Bar{h}(x_1,\cdots,x_j+L_j,\cdots,x_n)=\Bar{h}(x_1,\cdots,x_j,\cdots,x_n)$ for all $x \in \mathbb{R}^n$. Analogously, this holds for $\Bar{g}$. 
 
We define $QP(L)$ as the set of quasi-periodic functions associated with $L=(L_1,\cdots,L_n)$. The positive numbers $L_1,\cdots,L_n$ are referred to as quasi-periods. Therefore we will assume that  $h \in QP(L^h)$ and $g \in QP(L^g)$, where $L^h$ and $L^g$ are vectors in $\mathbb{R}^n$ and $\mathbb{R}^m$, respectively. 

It is noteworthy that when $n=1$ and $m=1$, the scenario reverts to the conventional periodic case.Furthermore, to denote periodic functions associated with quasi-periodic functions, we will use letters with bars.

It is not restrictive to assume that the associated frequencies to the quasi-periods $L=(L_1,\cdots,L_n)$ are linearly independet on $\mathbb{Z}$. Under this assumption, Kronecker's Lemma (see Appendix of \cite{Brai}) guarantees that $\bar{h}$ and $\bar{g}$ are uniquely determined by $h$ and $g$ respectively.

Using properties of quasi-periodic functions we obtain explicitly the homogenized limit problem of \eqref{transformed-problem-hat}  when the oscillating boundaries are given by \eqref{quasiper}.
We just need to prove that these particular functions $k^1_\eps$ and $k^2_\eps$ satisfy hypothesis {\bf (H)}.

First, notice that 
$$\eps\frac{\partial k^1_\eps}{\partial x}(x)=\eps^{1-\alpha}\frac{\partial h}{\partial x}\Big(\frac{x}{\eps^\alpha}\Big), \quad \eps\frac{\partial k^2_\eps}{\partial x}(x)=\eps^{1-\beta}\frac{\partial g}{\partial x}\Big(\frac{x}{\eps^\beta}\Big).$$
Then, since $0<\alpha, \beta<1$ we directly have {\bf (H.1)}.

As {\bf (H.2)} is immediately verified by hypothesis, we will focus on proving weak convergences {\bf (H.3)} and {\bf (H.4)} for quasi-periodic functions with multiple scales.

Taking into account the definition of $g$ and $h$ we can assume  that $K(x)=g(x) + h(x)$ is a quasi-periodic function with quasi-periods $L=(L_1,L_2,\cdots, L_{n+m})$ where the quasi-periods are not necessarily rationally independent. Therefore, there exists a $L-$periodic function $\bar K:\mathbb{R}^{n+m}\rightarrow \mathbb{R}$ such that $K$ is the trace of $\bar K$ in the sense of 
$$K(x)=\bar K(x,x,\cdots,x), \quad \forall x\in \mathbb{R}.$$

In addition, $K$ has an average, see Proposition 1.2 in \cite{BlancBris}, which is defined as follows
$${\displaystyle \mu(K)=\lim_{T\to \infty} \frac{1}{2T}\int_{-T}^{T} \big(g(y) + h(y)\big)\, dy= \frac{1}{|I(L)|}\int_{I(L)} \bar K(x_1,\cdots,x_{n+m})\, dx_1\cdots dx_{n+m},}$$
where $I(L)=(0,L_1)\times(0,L_2)\times \cdots \times (0,L_{n+m})$.

Since $K(\cdot)>0$ we can conclude that $\frac{1}{K}$ belongs to $QP(L)$ and its average is given by
$${\displaystyle \mu\Big(\frac{1}{K}\Big)=\lim_{T\to \infty} \frac{1}{2T}\int_{-T}^{T} \frac{1}{g(y) + h(y)}\, dy= \frac{1}{|I(L)|}\int_{I(L)} \frac{1}{\bar K(x_1,\cdots,x_{n+m})}\, dx_1\cdots dx_{n+m}.}$$

Below we show the convergence obtained for quasi-periodic functions with different oscillation scales.

\begin{proposition}
Let $F_\eps: \mathbb{R}\rightarrow \mathbb{R}$ be a quasi-periodic  function defined by
$$F_\eps(x)=\bar{F}\Big(\frac{x}{\eps^\alpha},\frac{x}{\eps^\alpha},\cdots\frac{x}{\eps^\alpha},\frac{x}{\eps^\beta}\cdots,\frac{x}{\eps^\beta}\Big),$$
where $\bar{F} \in \mathbb{R}^{n+m}$ is a $L-$periodic function. The following weak convergence holds
\begin{equation}\label{weakconv}
F_\eps \weto \mu(\bar{F}) \quad w-L^2(I).
\end{equation}
\end{proposition}
\begin{proof}
We just prove the result for $F_\eps$ a trigonometric polinomial of quasi-periods $L$  since the set of trigonometric polynomials of quasi-periods $L$ is dense en $QP(L)$ for the uniform norm, see \cite{BlancBris}.
Therefore, we have
$$F_{\eps}(x)=\sum_{k \in \mathbb{Z}^{n+m}}a_k\, e^{2\pi i \sum_{j=1}^n\frac{k_jx}{L_j \eps^\alpha}} e^{2\pi i \sum_{j=n+1}^{n+m}\frac{k_jx}{L_j \eps^\beta}},$$
where the sequence of $(a_k)_{k \in \mathbb{Z}^{N}}$ vanishes except for a finite number of values of $k$.
Notice that, for any interval $I=(a,b)$  we have
\begin{align*}
\int_a^b F_\eps(x)dx&=(b-a)a_0+\eps^{\alpha+\beta}\sum_{k\neq0}\frac{a_k e^{2\pi i \sum_{j=1}^n\frac{k_jb}{L_j \eps^\alpha}}e^{2\pi i \sum_{j=n+1}^{n+m}\frac{k_jb}{L_j \eps^\beta}}}{2i\pi(\eps^\beta \lambda_k+\eps^\alpha \mu_k)}\\
&-\eps^{\alpha+\beta}\sum_{k\neq 0}\frac{a_ke^{2\pi i \sum_{j=1}^n\frac{k_ja}{L_j \eps^\alpha}}e^{2\pi i \sum_{j=n+1}^{n+m}\frac{k_ja}{L_j \eps^\beta}}}{2i\pi(\eps^\beta \lambda_k+\eps^\alpha \mu_k)}\eto(b-a)a_0=(b-a)\mu(F),
\end{align*}
where $\lambda_k=\sum_{j=1}^n\frac{k_j}{L_j}$ y $\mu_k=\sum_{j=n+1}^n+m\frac{k_j}{L_j}$. 

Thus, for any piecewise constant compactly supported function $\varphi$ we obtain
$$\int F_\eps \varphi \eto \varphi \mu(F).$$
Using that piecewise constant functions are dense in $L^2(\mathbb{R})$  we have the result.
\end{proof}

Consequently, we obtain that {\bf (H.3)} and {\bf (H.4)} are satisfied for the particular case where both oscillating boundaries are given by quasi-periodic functions.

\begin{corollary}
Let $K_\eps(x)=h\big(\frac{x}{\eps^\alpha}\big)+g\big(\frac{x}{\eps^\beta}\big)$ be a quasi-periodic function with two different scales, we have
\begin{equation*}
K_\eps \weto \mu(K) \quad \text{w}-L^2(I), \quad \frac{1}{K_\eps}=\frac{1}{g\big(\frac{\cdot}{\eps^\beta}\big)+h\big(\frac{\cdot}{\eps^\alpha}\big)}\weto \mu\Big(\frac{1}{K}\Big) \quad \hbox{w}-L^2(I).
\end{equation*}
\end{corollary}

\begin{proof}
Trivial from the previous Theorem.
\end{proof}

Therefore, we are in conditions of Theorem \ref{main appen} and with the definition of $f_0$ given by \eqref{limit f1} and denoting by $\hat f=\frac{f_0}{\mu(g)+\mu(h)}$,  then we have  we obtain the following convergence result for the particular case of quasi-periodic functions:

$$\hat w^\eps \to \hat w, \hbox{ w}-H^1(I),$$ 
$$|||w^\eps - \hat w|||_{L^2(R^\eps)} \to 0,$$
 where   $\hat w \in H^1(I)$ is the weak solution of  the following Neumann problem
\begin{equation*}\label{homogenized problem1}
\left\{
\begin{gathered}
-\frac{1}{P{\big(\mu}(g)+{\mu}(h)\big)}  {\hat w}_{xx} + \hat w = \hat f, \quad x \in I, \\
\hat w'(0) = \hat w' (1) = 0,
\end{gathered}
\right.
\end{equation*}
where the constant $P$ is such that 
\begin{equation*}
 \frac{1}{h\Big(\frac{x}{\eps^\alpha}\Big) + g\Big(\frac{x}{\eps^\beta}\Big)} \weto P \quad \text{w}-L^2(I).
\end{equation*}
 Therefore $P$ is given by

$$
P
= \frac{1}{|I(L)|}\int_{I(L)} \frac{1}{F(x_1,\cdots,x_{m+n})}\, dx_1\cdots dx_{m+n}\,.$$

\begin{remark}
Notice what happend in the particular case where $g$ and $h$ are $L_1-$periodic and $L_2$-periodic respectively. In \cite{ArrVi2020} the authors have the following convergence
$$\frac{1}{G_\eps}=\frac{1}{g(\frac{\cdot}{\eps^\beta})+h(\frac{\cdot}{\eps^\alpha})}\weto  \frac{1}{p_0},$$
where $p_0$ is defined as follows
\begin{eqnarray*}
\frac{1}{p_0}
= 
 \left\{ 
 \renewcommand{\arraystretch}{2.5}
\begin{array}{ll}
{\displaystyle \lim_{T\to \infty} \frac{1}{T}\int_{0}^{T} \frac{1}{g(y) + h(y)}\, dy, \quad \hbox{if } \alpha=\beta,}  \\
 {\displaystyle\frac{1}{L_1L_2}\int_{0}^{L_1}\int_{0}^{L_2}\frac{1}{g(y) + h(z)}\, dz dy, \quad \hbox{if } \alpha \neq \beta.}
\end{array}
\right.
\end{eqnarray*}

Notice that this result is in complete accordance with the previous proposition. For this particular case, it is found that $P=\frac{1}{p_0}$. In fact, for the case $\alpha\neq \beta$ and for $\alpha=\beta$ we have:
\begin{align*}
\displaystyle P&=\mu\Big(\frac{1}{g+h}\Big)=\lim_{T\to \infty} \frac{1}{T}\int_{0}^{T} \frac{1}{g(y) + h(y)}\, dy\\
&=\frac{1}{L_1L_2}\int_{0}^{L_1}\int_{0}^{L_2}\frac{1}{g(y) + h(z)}\, dz dy=\frac{1}{p_0}.
\end{align*}
\end{remark}

We can also write Theorem \ref{main appen} for almost periodic functions in the sense of Besicovitch, see \cite{Be}. We just have to take into account that
the set of almost periodic functions in $\mathbb{R}$ is the closure of the set of trigonometric polynomial for the the mean square norm (or Besicovitch norm), defined for a $P_n$ trigonometric polynomial as follows:
$$
\|P_n\|_2 = \left( \limsup_{T \to \infty} \frac{1}{2T} \int_{-T}^{T} |P_n(t)|^2 \, dt \right)^{1/2}.$$

Therefore, it is obvious, by means of the previous approximation, that any almost periodic oscillatory functions satisfy hipothesis {\bf (H)}. Therefore, if the thin domain is given by two almost periodic functions $g$ and 
$h$ as follows:  
$$R^\epsilon = \Big\{ (x, y) \in \R^{2} \; | \;  x \in I \subset \R,  \; -\eps g(x/\eps^\alpha)< y < \epsilon h(x/\eps^\beta) \Big\}.$$
Then, in condictions of Theorem \ref{main appen}  and denoting by $\hat f=\frac{f_0}{\mu(g)+\mu(h)}$, we can guarantee:

$$\hat w^\eps \to \hat w, \hbox{ w}-H^1(I),$$ 
$$|||w^\eps - \hat w|||_{L^2(R^\eps)} \to 0,$$
 where   $\hat w \in H^1(I)$ is the weak solution of  the following Neumann problem
\begin{equation*}\label{homogenized problem1}
\left\{
\begin{gathered}
-\frac{1}{P{\big(\mu}(g)+{\mu}(h)\big)}  {\hat w}_{xx} + \hat w = \hat f, \quad x \in I, \\
\hat w'(0) = \hat w' (1) = 0,
\end{gathered}
\right.
\end{equation*}
where the constant $P=\limsup_{T \to \infty} \frac{1}{2T} \int_{-T}^{T} \frac{1}{g(t)+h(t)} \, dt.$

\section{Numerical evidences}
In this section, we numerically investigate the behavior of the solutions to equation \eqref{1OPI0} as \(\varepsilon\) approaches zero, where the thin domain is defined by the graph of a quasi-periodic function, see Figure \ref{exthin1f}.

Then, we consider the following particular thin domain:

\begin{equation}\label{exthin1}
\begin{split}
R^\epsilon = \Big\{ (x, y) \in \R^{2} \; | \;  x \in (0,20), \; &-\eps\Big(8-\sin\Big(\frac{x}{\sqrt[5]{\eps}}\Big)-\sin\Big(\frac{x\pi}{8\sqrt[5]{\eps}}\Big)\Big)< y \\
&< \epsilon\Big(8+\sin\Big(\frac{x}{\sqrt[5]{\eps}}\Big)+\sin\Big(\frac{x\pi}{8\sqrt[5]{\eps}}\Big)\Big) \Big\}.
\end{split}
\end{equation}

 Moreover, the narrow strip $\theta^\epsilon$ is defined by 

\begin{align*}
\theta^\epsilon = \Big\{ (x, y) \in \R^2 \; | \;  \;  x \in (0,20), \; &\epsilon\Big(8+\sin\Big(\frac{x}{\sqrt[5]{\eps}}\Big)+\sin\Big(\frac{x\pi}{8\sqrt[5]{\eps}}\Big) -\sqrt[18]{\eps}\Big(2+\sin\Big(\frac{x}{\sqrt[3]{\eps}}\Big)\Big)\Big)<y\\
&< \epsilon\Big(8+\sin\Big(\frac{x}{\sqrt[5]{\eps}}\Big)+\sin\Big(\frac{x\pi}{8\sqrt[5]{\eps}}\Big)\Big) \Big\}.
\end{align*}

\begin{figure}[!htp]
\begin{center}
\includegraphics[scale=0.3]{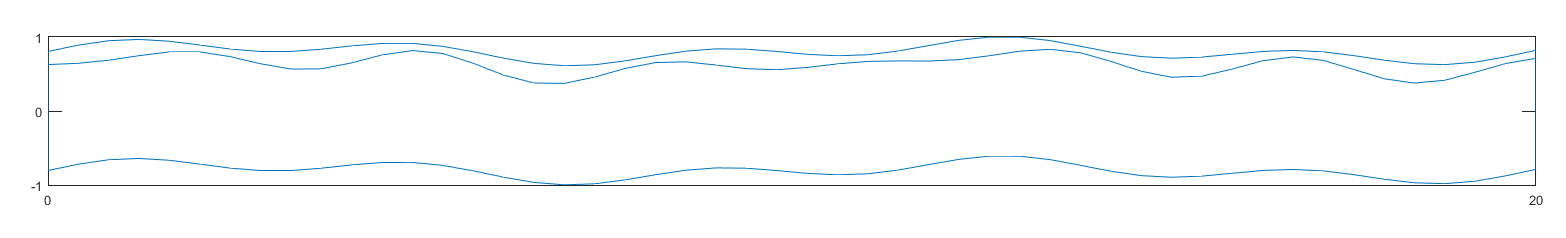}
\caption{Thin domain $R^\epsilon$ featuring the slender strip $\theta^\epsilon$ for $\epsilon=1$. }
\label{exthin1f}
\end{center}
\end{figure}

The problem was discretized using a triangular mesh that is finer in the narrow strip, see Figure \ref{mesh}.

\begin{figure}[!htp]
\centering
\includegraphics[scale=0.3]{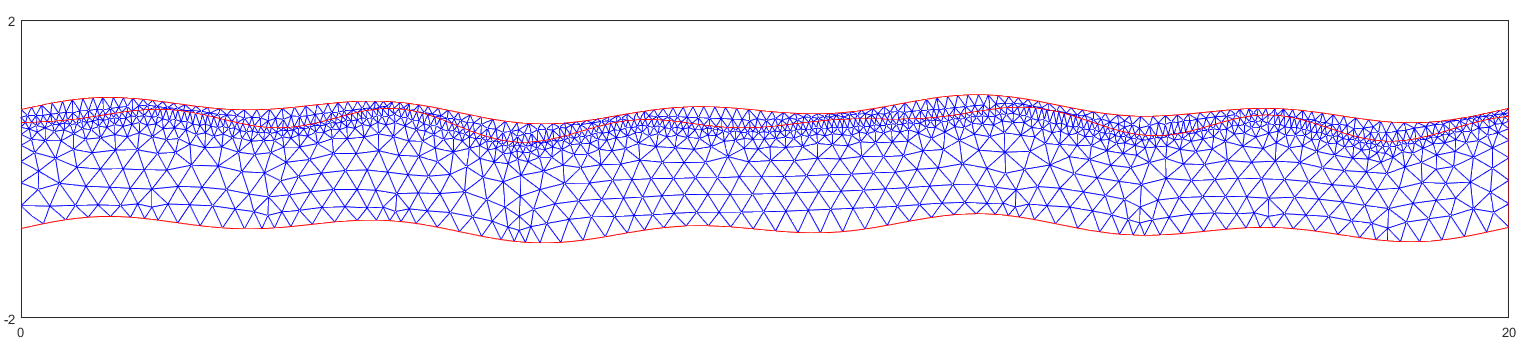}
\caption{Triangular mesh for the slender domain with finer density in the strip.}
\label{mesh}
\end{figure}

We analyze the behavior of the solutions as $\eps$ tends to zero taking the forcing term $f$ as $f(x)=1+sin(x)$. 

\begin{figure}[!htp]
  \centering
  \begin{subfigure}{0.35\textwidth}
    \centering
    \includegraphics[width=\linewidth]{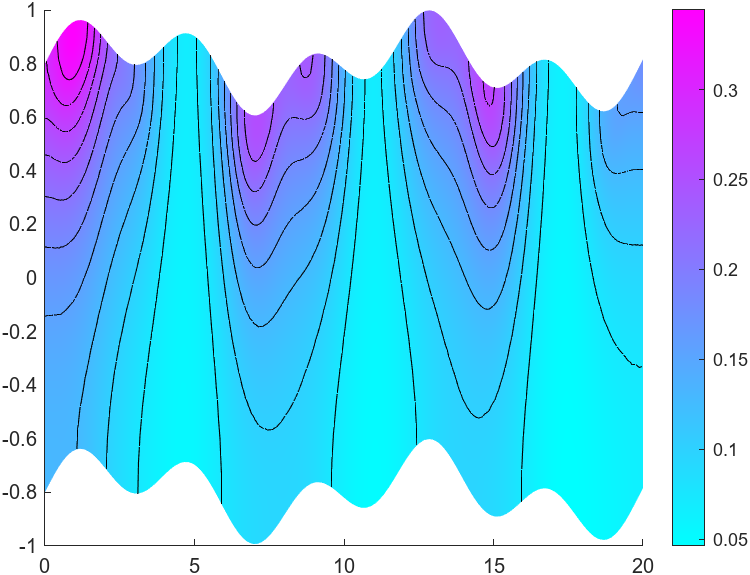}
    \caption{$\varepsilon=0.1$}
  \end{subfigure}
  \hfill
  \begin{subfigure}{0.55\textwidth}
    \centering
    \includegraphics[width=\linewidth]{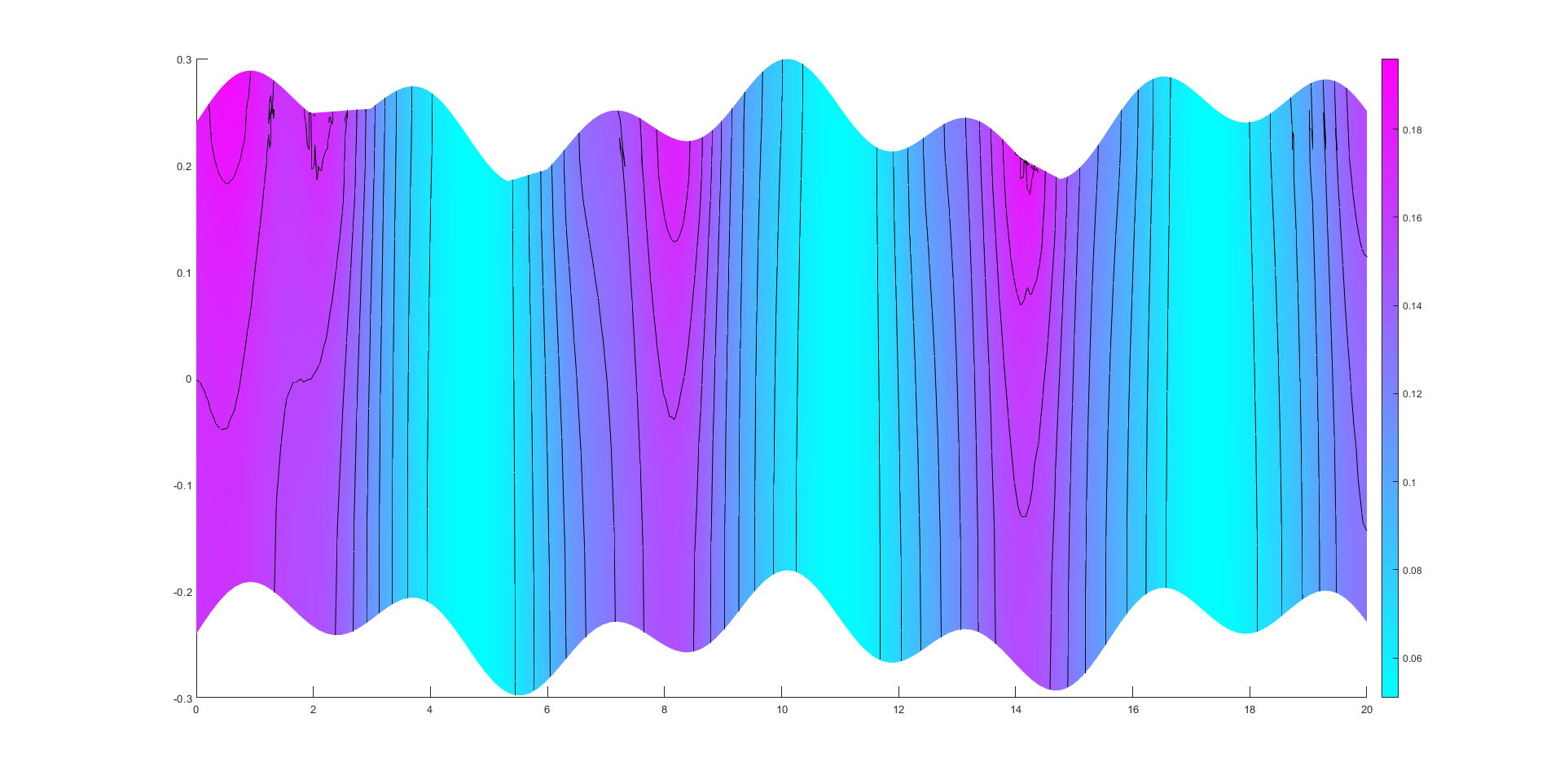}
    \caption{$\varepsilon=0.04$}
  \end{subfigure}
  
  \vspace{0.5cm}
  
  \begin{subfigure}{0.6\textwidth}
    \centering
    \includegraphics[width=\linewidth]{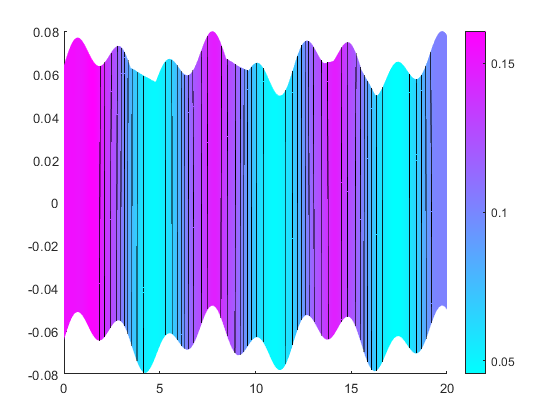}
    \caption{$\varepsilon=0.0008$}
  \end{subfigure}
  \caption{Contour levels of solutions for some values of $\eps$.}
  \label{contoursol}
\end{figure}

First we show a color map of the solutions with the corresponding contours levels for different values of $\eps$. For higher values of $\epsilon$, the dependence of the solutions on two dimensions can be seen, with significant variations in $y$ due to the applied force in the boundary strip.It can be clearly observed, Figure \ref{contoursol}, that as $\epsilon$ becomes small, the dependence of the vertical variable on the solutions disappears. This is consistent with the fact that the domain is thinning out, progressively shrinking in the vertical direction. This fact can also be observed in Figure \ref{u3d} where the solution was represented for two values of $\epsilon$.
While for $\epsilon=0.2$, it is clear that the solution undergoes dramatic changes from $y=-2$ to $y=2$, for $\epsilon=0.05$, it essentially maintains the same profile.

\begin{figure}[!htp]
  \centering
  \begin{subfigure}{0.45\textwidth}
    \centering
    \includegraphics[width=\linewidth]{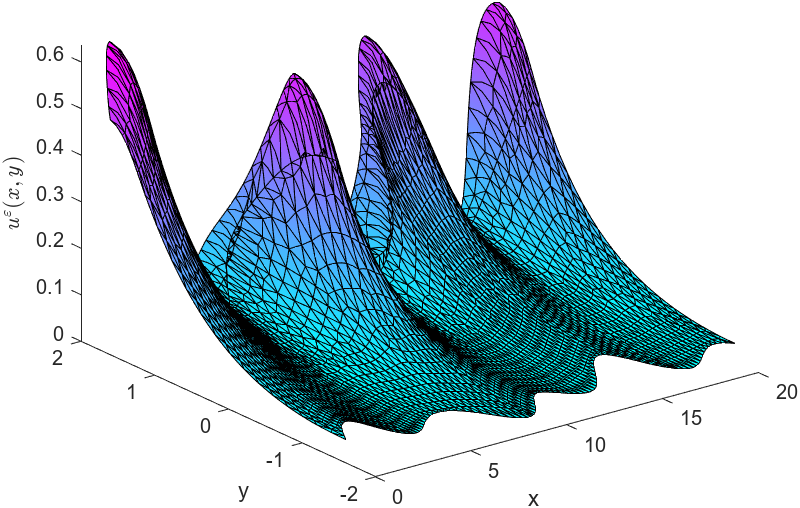}
    \caption{$\varepsilon=0.2$}
    \label{fig:subfig1}
  \end{subfigure}
  \hfill
  \begin{subfigure}{0.45\textwidth}
    \centering
    \includegraphics[width=\linewidth]{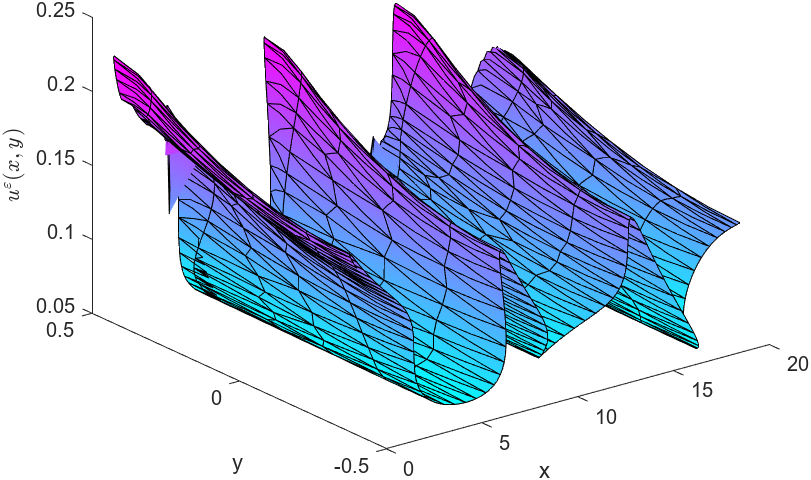}
    \caption{$\varepsilon=0.05$}
    \label{fig:subfig2}
  \end{subfigure}
   \caption{Solution for some values of $\eps$.}
  \label{u3d}
\end{figure}

This phenomenon becomes more apparent in Figure \ref{comparison0}, where slices of the solutions at $y=-0.2, 0, 0.2$ for various $\epsilon$ values are displayed. It is observed that as $\epsilon$ decreases, the similarity among the slices increases significantly.

\begin{figure}[!htp]
  \centering
  \begin{subfigure}{0.22\textwidth} 
    \centering
    \includegraphics[width=\linewidth]{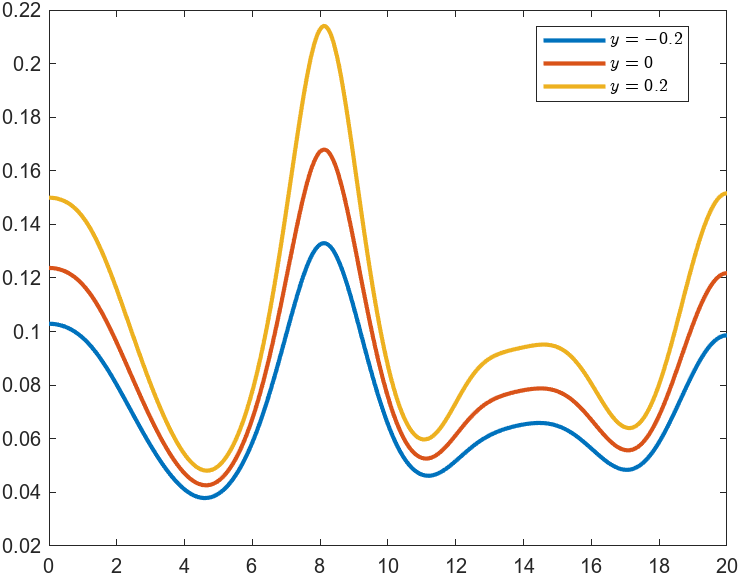}
    \caption{$\varepsilon=0.2$}
  \end{subfigure}
  \quad 
  \begin{subfigure}{0.22\textwidth}
    \centering
    \includegraphics[width=\linewidth]{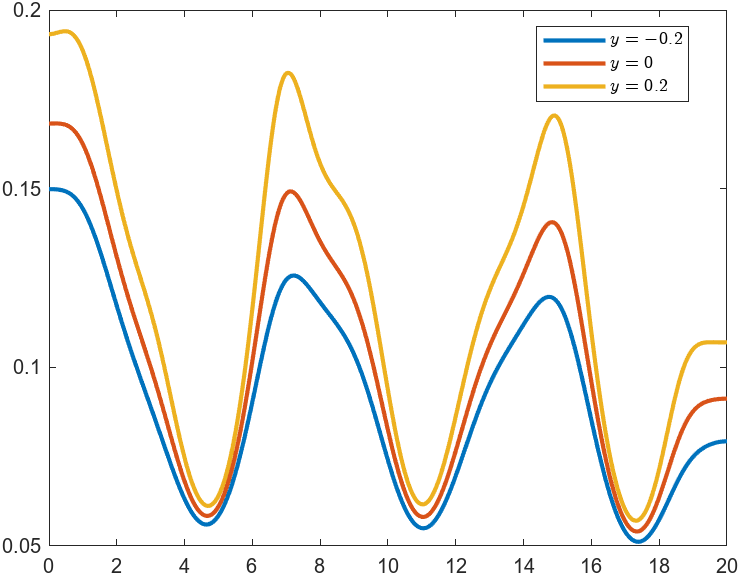}
    \caption{$\varepsilon=0.1$}
  \end{subfigure}
  
  \vspace{0.5cm} 
  
  \begin{subfigure}{0.22\textwidth}
    \centering
    \includegraphics[width=\linewidth]{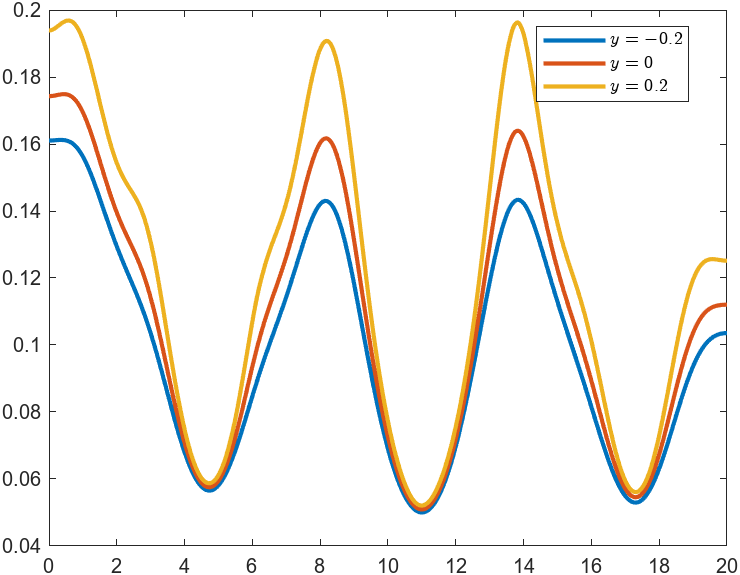}
    \caption{$\varepsilon=0.07$}
  \end{subfigure}
  \quad 
  \begin{subfigure}{0.22\textwidth}
    \centering
    \includegraphics[width=\linewidth]{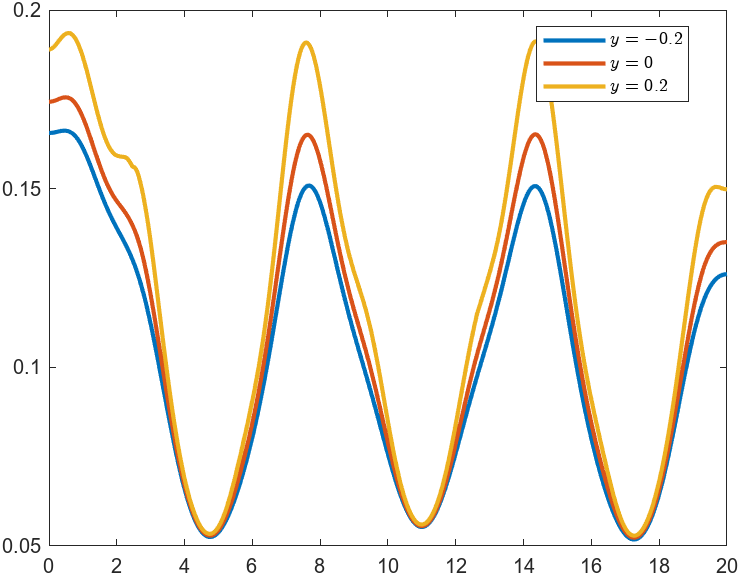}
    \caption{$\varepsilon=0.05$}
  \end{subfigure}
  \caption{Solution for different values of $y$.}
  \label{comparison0}
\end{figure}

This leads us to attempt to compare the solution for increasingly smaller values of $\epsilon$ with the solution of the limit problem. In Figure \ref{comparisonlimit}, we represent the solutions for different values of $\epsilon$ with the vertical variable fixed at $y=0$, along with the solution to the limit equation. It appears to be inferred that, as expected, as $\epsilon$ becomes smaller, the solution of the problem in the thin domain increasingly resembles the limit solution.

\begin{figure}[!htp]
\centering
\includegraphics[scale=0.6]{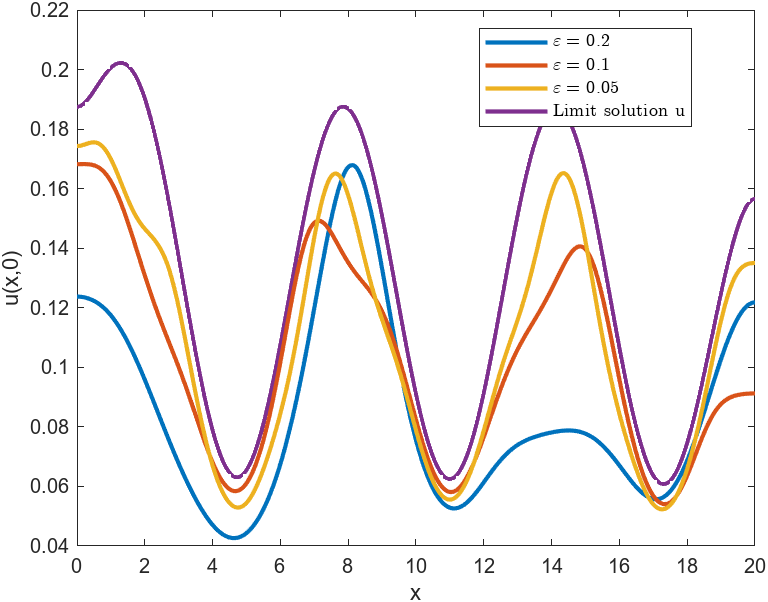}
\caption{$u^\eps(x,0)$ and the limit solution $u(x)$. }
\label{comparisonlimit}
\end{figure}

Since Theorem \ref{main appen} establishes strong convergence of the solutions to the limit problem's solution, we have also examined the $L^2$ norm of the difference for various values of $\epsilon$. In particular, for $\eps=0.1$ we get that the norm of the difference in the thin domain is $1.7853$, for $\eps=0.08$ is $1.032052$ and if $\eps=0.04$ the norm of the difference is $0.467611$. Therefore,  the error norm appears to decrease linearly.

A natural question is whether such approximation results can be improved in order to describe the
asymptotic behavior of the Dynamical System generated by the parabolic equation associated with \eqref{1OPI0}
posed in more general thin regions of $\R^n$. It is our goal to investigate this question in a forthcoming paper.

\printbibliography

\medskip

\end{document}